\documentclass[12pt]{amsart}
\usepackage{amsfonts}
\usepackage{amsmath}
\usepackage{amssymb}
\usepackage[margin=1in]{geometry}
\usepackage{amsthm}   							
\usepackage{mathtools}							
\usepackage{commath}																		
																						
\usepackage[english]{babel}

\makeatletter
\@namedef{subjclassname@2020}{\textup{2020} Mathematics Subject Classification}
\makeatother

\newtheorem{theorem}{Theorem}[section]
\newtheorem{corollary}[theorem]{Corollary}
\newtheorem{lemma}[theorem]{Lemma}

\newtheorem{conjecture}[theorem]{Conjecture}

\theoremstyle{definition}

\newtheorem{remark}[theorem]{Remark}

\numberwithin{equation}{section}

\newcommand{\N}{\mathbb{N}}

\newcommand{\R}{\mathbb{R}}

\renewcommand{\Re}{\operatorname{Re}}
\renewcommand{\Im}{\operatorname{Im}}

\newcommand{\I}{\mathrm{i}}
\newcommand{\e}{\mathrm{e}}

\newcommand{\eps}{\varepsilon}
\newcommand{\vphi}{\varphi}

\newcommand{\MP}{\mathcal{P}}
\newcommand{\MN}{\mathcal{N}}

\newcommand{\zetaDMV}{\zeta_{\mathrm{DMV}}}

\DeclareMathOperator{\Li}{Li}

\begin{document}

\title{On zero-density estimates for Beurling zeta functions}

\author[F. Broucke]{Frederik Broucke}

\address{Department of Mathematics: Analysis, Logic and Discrete Mathematics\\ Ghent University\\ Krijgslaan 281\\ 9000 Gent\\ Belgium}

\email{fabrouck.broucke@ugent.be}



\begin{abstract}
We show the zero-density estimate
\[	
	N(\zeta_{\MP}; \alpha, T) \ll T^{\frac{4(1-\alpha)}{3-2\alpha-\theta}}(\log T)^{9}
\]
for Beurling zeta functions $\zeta_{\MP}$ attached to Beurling generalized number systems with integers distributed as $N_{\MP}(x) = Ax + O(x^{\theta})$. We also show a similar zero-density estimate for a broader class of general Dirichlet series, consider improvements conditional on finer pointwise or $L^{2k}$-bounds of $\zeta_{\MP}$, and discuss some optimality questions. 
\end{abstract}
\maketitle

\section{Introduction}
A \emph{Beurling generalized number system} is a pair $(\MP, \MN)$, where $\MP = (p_{j})_{j}$, the sequence of \emph{generalized primes}, consists of real numbers $p_{j}$ with $1 < p_{1} \le p_{2} \le \dotsb$ and $p_{j}\to \infty$,  and $\MN=(n_{j})_{j}$, the sequence of \emph{generalized integers}, consists of all possible products of the generalized primes, including $1$. One also orders this sequence in a non-decreasing fashion. We allow repetitions in both sequences; in the sequence $\MN$, a real number occurs as many times as its number of distinct factorizations into generalized primes.

This notion was introduced by Beurling \cite{Beurling}, who wanted to investigate the minimal assumptions needed to show the \emph{prime number theorem}. Denoting $\pi_{\MP}(x) = \sum_{p_{j}\le x}1$ and $N_{\MP}(x) = \sum_{n_{j}\le x}1$, his theorem states the asymptotic $N_{\MP}(x) = Ax + O(x(\log x)^{-\gamma})$ where $A$ is some positive constant and $\gamma > 3/2$, yields the prime number theorem in the form
\begin{equation}
\label{PNT}
	\pi_{\MP}(x) \sim \Li(x).
\end{equation}
Here, $\Li(x) = \int_{2}^{x}(\log u)^{-1}\dif u$ stands for the logarithmic integral.

We consider Beurling generalized number systems $(\MP, \MN)$ for which the integers are \emph{well-behaved}, meaning that there exists a number $\theta\in[0,1)$ for which
\begin{equation}
\label{well-behaved}
	N_{\MP}(x) = Ax + O(x^{\theta}),
\end{equation}
for some constant $A>0$, the \emph{density} of the integers. If we want to mention the exponent $\theta$, we will say the the integers are $\theta$-well-behaved. This assumption is quite natural, for instance it readily yields the analytic continuation of $\zeta_{\MP}(s) - \frac{A}{s-1}$ to the half-plane $\Re s > \theta$. Here $\zeta_{\MP}(s)$ is the \emph{Beurling zeta function} associated with the system $(\MP, \MN)$:
\[
	\zeta_{\MP}(s) = \sum_{j=1}^{\infty}\frac{1}{n_{j}^{s}} = \prod_{j=1}^{\infty}\frac{1}{1-p_{j}^{-s}}, \quad \Re s > 1.
\]

Besides being of independent interest, the class of Beurling zeta functions coming from well-behaved integers is also an interesting ``testing ground'' to analyze the strength of the analytic techniques used occurring in the study of the Riemann zeta function in the critical strip and the distribution of prime numbers. For example, the results of this paper can be regarded as a study on how far the classical zero-detection methods can be pushed in a general setting: our results apply to a very broad class of (general) Dirichlet series, which also includes many of the classical $L$-functions, see Theorem \ref{general theorem} and Remark \ref{general coefficients} below. 

\medskip
Our main goal is to obtain \emph{zero-density estimates}. These are bounds for the number of zeros of $\zeta_{\MP}(s)$ in rectangles: let $\alpha>\theta$, $T>0$, then we denote
\[
	N(\zeta_{\MP}; \alpha, T) = \#\{\rho = \beta+\I\gamma: \zeta_{\MP}(\rho)=0, \beta\ge\alpha, |\gamma|\le T\}.
\]
We often simply write $N(\alpha, T)$ if the Beurling zeta function is clear from the context. Classically, one can use such zero-density estimates to study the distribution of prime numbers in short intervals. This application is also present in the Beurling case. We refer to \cite[Section 4]{BrouckeDebruyne} for more details.

From the Euler product representation of $\zeta_{\MP}(s)$, it is immediate that $\zeta_{\MP}(s)$ has no zeros with $\beta>1$. Actually, the well-havedness of the integers \eqref{well-behaved} implies the zero-free region
\[
	\sigma > 1- \frac{C(1-\theta)}{\log \abs{t}}, \quad \abs{t}\ge T_{0},
\]
for some absolute constant $C$ and some $T_{0}$ depending on the system $(\MP, \MN)$. This was first shown by Landau \cite{Landau} for Dedekind zeta functions attached to number fields, but his method readily applies to the Beurling case.

We are thus interested in bounds of $N(\alpha, T)$ in the critical range $\theta < \alpha < 1$.
The local number of zeros can be bounded by an application of Jensen's formula, see e.g.\ Lemma \ref{local zeros F} below, or \cite[Lemma 3.3]{ReveszvonMangoldt}.
\begin{lemma}
\label{local number of zeros}
Let $\zeta_{\MP}(s)$ be the Beurling zeta function associated with a number system $(\MP, \MN)$ with $\theta$-well-behaved integers. For each $T>2$ and $\alpha > \theta$, the number of zeros $\rho = \beta+\I \gamma$ of $\zeta_{\MP}(s)$ with $\beta\ge\alpha$ and $\abs{\gamma}\in [T, T+1]$ is bounded by $O_{\alpha}(\log T)$.
\end{lemma}
From this lemma we see that $N(\alpha, T) \ll_{\alpha} T\log T$. The goal of zero-density estimates is to improve upon this ``trivial'' estimate. 
More precisely, our goal is to find an abscissa $\sigma_{0} \in (\theta,1)$  and a function $c(\alpha)<1$ for which a bound of the form
\[
	N(\alpha, T) \ll T^{c(\alpha)}(\log T)^{O(1)} \quad \text{or} \quad N(\alpha, T) \ll_{\eps} T^{c(\alpha)+\eps}, \text{ for every } \eps>0,
\]
holds uniformly for $\alpha>\sigma_{0}$. 

For zeta functions of well-behaved number systems, a zero-density estimate was first shown by Sz.\ R\'ev\'esz in \cite{Reveszdensity} under the additional assumptions $\MN \subseteq \N$ (the so-called integrality assumption) and a sort of ``Ramanujan condition'' expressing that $\MN$ does not have too many repeated values. A little later, these restrictions could be removed, as found independently by R\'ev\'esz in \cite{ReveszCarlson} and by G.\ Debruyne and the author in \cite{BrouckeDebruyne}, where the admissible values
\begin{equation}
\label{previous c}
	\sigma_{0} = \frac{11+\theta}{12}, \quad c(\alpha) = \frac{12(1-\alpha)}{1-\theta}, \quad \text{and} \quad \sigma_{0} = \frac{2+\theta}{3}, \quad c(\alpha) = \frac{c_{1}(\alpha)(1-\alpha)}{1-\theta},
\end{equation}
respectively, with $c_{1}(\alpha)$ an increasing function with $c_{1}(\frac{2+\theta}{3})=3$ and $c_{1}(1) = 4$, were established. With the integrality assumption and the Ramanujan condition, the value of $c(\alpha)$ can be lowered further when $\alpha$ is sufficiently close to $1$, as shown by J.\ Pintz and R\'ev\'esz in the recent preprint \cite{PintzRevesz}. For example, they obtained $N(\alpha, T) \ll_{\eps} T^{\frac{2(1-\alpha)}{\alpha-\theta}+\eps}$ when $\alpha$ is sufficiently close to $1$.

In this paper, we keep the general setup of $\theta$-well-behaved integers without imposing additional assumptions. We obtain a zero-density estimate in the widest possible range $\alpha >\sigma_{0}=\frac{1+\theta}{2}$ and also lower the value of the function $c(\alpha)$ in \eqref{previous c}.
Our main result is
\begin{theorem}
\label{zero-density theorem}
Let $(\MP, \MN)$ be a Beurling number system satisfying \eqref{well-behaved} for some $\theta \in [0,1)$. Then we have uniformly for $\alpha \ge \frac{1+\theta}{2}$ and $T \ge 2$, 
\begin{equation}
\label{main zero-density estimate}
	N(\zeta_{\MP}; \alpha, T) \ll T^{\frac{4(1-\alpha)}{3-2\alpha-\theta}}(\log T)^{9}.
\end{equation}
\end{theorem}
Applying this to the system $(\mathbb{P}, \N)$ of rational primes and integers, for which we can take $\theta = 0$, we obtain the estimate $N(\zeta; \alpha, T) \ll T^{\frac{4(1-\alpha)}{3-2\alpha}}(\log T)^{9}$ for the number of zeros of the Riemann zeta function. This estimate was first shown by Titchmarsh \cite{Titchmarsh30} (with the power of $\log T$ replaced by $T^{\eps}$), whose proof uses the approximate functional equation of the Riemann zeta function, which is in general not available for Beurling zeta functions. 

\medskip

Our proof of Theorem \ref{zero-density theorem} does not require any (approximate) functional equation and actually applies to quite a broad class of general Dirichlet series $F(s) = \sum_{j}a_{j}n_{j}^{-s}$. An important aspect in the proof is the existence of the inverse of $F$, $G(s) = 1/F(s)$ as a general Dirichlet series, so it is sensible to assume that $(n_{j})_{j}$ is a sequence of Beurling generalized integers. In fact, if it is not, we can without loss of generality consider $F$ as a Dirichlet series over the multiplicative semigroup generated by $(n_{j})_{j}$.
\begin{theorem}
\label{general theorem}
Let $(n_{j})_{j\ge1}$ (with $n_{1}=1$) be a sequence of Beurling generalized integers satisfying $N(x+x^{\theta}) - N(x-x^{\theta}) \ll_{\eps} x^{\theta+\eps}$ for some $\theta\in [0,1)$ and for every $\eps>0$. Let
\[
	F(s) = 1 + \sum_{j=2}^{\infty}\frac{a_{j}}{n_{j}^{s}}, \quad G(s) = \frac{1}{F(s)} = 1 + \sum_{j=2}^{\infty}\frac{b_{j}}{n_{j}^{s}}
\]
be (general) Dirichlet series satisfying $a_{j} \ge 0$ and the ``Ramanujan bound'' $a_{j} \ll_{\eps} n_{j}^{\eps}$, and $\abs{b_{j}} \ll_{\eps} n_{j}^{\eps}$ for every $\eps>0$. Suppose further that 
\begin{equation}
\label{A(x)}
	A(x) \coloneqq \sum_{n_{j}\le x}a_{j} = xP(\log x) + O\bigl(x^{\theta}(\log x)^{q}\bigr)
\end{equation}
for a certain polynomial $P$ of degree $d$ and some $q\in \R$. Then $F(s)$ has meromorphic continuation to $\Re s > \theta$ whose only singularity is a pole of order $d$ at $s=1$, and
\[
	N(F; \alpha, T) \ll_{\eps} T^{\frac{4(1-\alpha)}{3-2\alpha-\theta}+\eps}; \quad \text{for every } \eps > 0,
\]
uniformly for $\alpha \ge \frac{1+\theta}{2}$.
\end{theorem}
It is an interesting question whether some specific structure of the Beurling zeta function $\zeta_{\MP}(s)$, e.g.\ the fact that $a_{j} \equiv1$, or $\zeta_{\MP}$ having an Euler product, or simply having multiplicative coefficients, can be leveraged in some way to obtain a better zero-density estimate, or whether the number of zeros of $\zeta_{\MP}(s)$ can in general be as large (up to a factor $T^{\eps}$) as the number of zeros of more general Dirichlet series $F(s)$.

If more analytic information of $\zeta_{\MP}(s)$ is available, then Theorem \ref{zero-density theorem} can be improved. We investigate the case of higher moment estimates or pointwise bounds on the line $\sigma=\frac{1+\beta}{2}$, resulting in Theorems \ref{zero-density moments} and \ref{zero-density subconvexity}, respectively.

\medskip
The paper is organized as follows. First we discuss some preliminary results in Section \ref{sec: preliminary results}. The two main technical ingredients are a mean value theorem for general Dirichlet polynomials and a second moment estimate for Beurling zeta functions. Next, we prove Theorem \ref{zero-density theorem} in Section \ref{sec: proof}. We follow this up with Section \ref{sec: general}, where we discuss the proof of Theorem \ref{general theorem}. In Section \ref{sec: conditional improvements} we prove the conditional improvements of Theorem \ref{zero-density theorem}, namely Theorems \ref{zero-density moments} and \ref{zero-density subconvexity}. We also discuss a hypothetical improved mean value theorem in the style of Montgomery's conjecture for Dirichlet polynomials.
As alluded to, the range of validity $\alpha > \frac{1+\theta}{2}$ of density estimates is optimal. When $\theta=0$, this follows simply from considering the Riemann zeta function. In the final section, we indicate why this range should be optimal for any value of $\theta$, and prove this when $\theta>1/2$. We also make some comments on the question of the optimal value of $c(\alpha)$.

\section{Preliminary results}
\label{sec: preliminary results}
The proof of Theorem \ref{zero-density theorem} rests upon two main results: a mean value theorem for general Dirichlet polynomials, and a mean value theorem for Beurling zeta functions.

The classical mean value theorem (see e.g. \cite[Theorem 6.1]{Montgomery}) for ordinary Dirichlet polynomials $\sum_{n}a_{n}n^{-\I t}$ relies on the separation of the frequencies $\log n$, $n\in \N$ (see also \cite[Corollary 2]{MontgomeryVaughan}). A generalization of this result was established in \cite{BrouckeDebruyne}, where the lack of separation of general frequencies $\log n_{j}$ was remedied by introducing a factor $\chi$, which counts the size of ``local clusters'' of $n_{j}$'s. Theorem 2.1 of \cite{BrouckeDebruyne} states:
\begin{theorem}
\label{MVT Dirichlet polynomials}
Let $N\ge2$ and suppose that $1\le n_{1} \le n_{2} \le \dotso \le n_{J} \le N$ are real numbers. For $\lambda > 0$ denote by $\chi(x, \lambda)$ the number of $n_{j}$'s within distance $\lambda$ of $x$:
\[
	\chi(x,\lambda) = \#\{j: |x-n_{j}| \le \lambda\}.
\]
Suppose $a_{j}$ $(j=1, \dotsc, J)$ are complex numbers, and set
\[
	S(\I t) = \sum_{j=1}^{J}\frac{a_{j}}{n_{j}^{\I t}}.
\]
Then for $T_{0}\in\R$, $T>0$, and $\eta>0$ we have
\[
	\int_{T_{0}}^{T_{0}+T}\abs{S(\I t)}^{2}\dif t \ll \biggl(T + \frac{1}{\eta}\biggr)\sum_{j=1}^{J}\chi(n_{j},\eta N)\abs{a_{j}}^{2}.
\]
\end{theorem}

Estimating a sum over well-spaced points by an integral, one obtains the following discrete variant, see \cite[Theorem 2.2]{BrouckeDebruyne}.
\begin{theorem}
\label{discrete MVT T}
With the same notation as in Theorem \ref{MVT Dirichlet polynomials}, let $\delta>0$, and let $\mathcal{T} \subset [T_{0}+\delta/2, T_{0}+T-\delta/2]$ be a set of $\delta$-well-spaced points, meaning that $|t-t'| \ge \delta$ whenever $t, t'\in \mathcal{T}$ with $t\neq t'$. Then
\[
	\sum_{t\in\mathcal{T}}\abs{S(\I t)}^{2} \ll \biggl(T+\frac{1}{\eta}\biggr)\biggl(\log N + \frac{1}{\delta}\biggr)\sum_{j=1}^{J}\chi(n_{j}, \eta N)\abs{a_{j}}^{2}.
\]
\end{theorem}
We shall also need an estimate for the sum over complex numbers $s = \sigma + \I t$ with well-spaced imaginary parts, but for which the real parts need not be the same.
\begin{theorem}
\label{discrete MVT S}
Let $\alpha>0$, $\delta>0$, and $\mathcal{S}$ a set of complex numbers $s=\sigma+\I t$ satisfying
\[
	 \sigma \ge \alpha,  \quad T_{0}+\delta/2 \le t \le T_{0} + T-\delta/2,
\]	 
and with the property that $\forall s, s'\in\mathcal{S}$: $s\neq s' \implies \abs{\Im s -\Im s'} \ge \delta$.
 
Suppose that $N\ge2$, $N\le n_{1} \le \dotso \le n_{J} \le 2N$, and let $a_{j}$ $(j=1, \dotsc, J)$ be complex numbers. Write $S(s) = \sum_{j=1}^{J}a_{j}n_{j}^{-s}$. Then for every $\eta>0$
\[
	\sum_{s\in\mathcal{S}}\abs{S(s)}^{2} \ll \biggl(T+ \frac{1}{\eta}\biggr)\biggl(\log N + \frac{1}{\delta}\biggr)\sum_{j=1}^{J}\chi(n_{j}, \eta N)\frac{\abs{a_{j}}^{2}}{N^{2\alpha}}.
\]
\end{theorem}
\begin{proof}
Write 
\[
	S(s, x) = \sum_{N\le n_{j}\le x}\frac{a_{j}}{n_{j}^{s}}.
\]
Integrating by parts we get for $s\in \mathcal{S}$
\begin{align*}
	S(s) 	&= \int_{N-}^{2N}x^{\alpha-\sigma}\dif S(\alpha+\I t, x) = (2N)^{\alpha-\sigma}S(\alpha+\I t) + (\sigma-\alpha)\int_{N}^{2N}x^{\alpha-\sigma-1}S(\alpha+\I t, x)\dif x \\
		&\ll \abs{S(\alpha+\I t)} + \frac{1}{N}\int_{N}^{2N}\abs{S(\alpha+\I t, x)}\dif x.
\end{align*}
By Cauchy--Schwarz,
\[
	\abs{S(s)}^{2} \ll \abs{S(\alpha+\I t)}^{2} + \frac{1}{N}\int_{N}^{2N}\abs{S(\alpha+\I t, x)}^{2}\dif x.
\]
We now apply Theorem \ref{discrete MVT T} to the sums $S(\alpha+\I t, x)$ with coefficients $b_{j} = a_{j}n_{j}^{-2\alpha}$ if $N\le n_{j}\le x$, and $b_{j}=0$ else.
\end{proof}
\medskip

The second ingredient is a mean value theorem for the Beurling zeta function. Recently in \cite{BrouckeHilberdink}, the maximal rate of growth of the \emph{mean square} for a general class of Dirichlet series was established. Applying Theorem 1.1 and Theorem 1.1+ of \cite{BrouckeHilberdink} to $\zeta_{\MP}(s)$ and $-\zeta_{\MP}'(s)$ respectively, we obtain:
\begin{theorem}
\label{MVT zeta}
Let $\zeta_{\MP}(s)$ be the Beurling zeta function associated to a number system $(\MP, \MN)$ with $\theta$-well-behaved integers. Then
\[
	\frac{1}{T}\int_{0}^{T}\abs[2]{\zeta_{\MP}\Bigl(\frac{1+\theta}{2}+\I t\Bigr)}^{2}\dif t \ll \log T, 
	\quad \frac{1}{T}\int_{0}^{T}\abs[2]{\zeta'_{\MP}\Bigl(\frac{1+\theta}{2}+\I t\Bigr)}^{2}\dif t \ll (\log T)^{3}.
\]
\end{theorem}

\section{Proof of the zero-density estimate}
\label{sec: proof}

We now prove Theorem \ref{zero-density theorem}, following the classical procedure from \cite[Chapter 12]{Montgomery}. In what follows, the implicit constants in the $O$-estimates may depend on the number system $(\MP, \MN)$, but are independent of $\alpha$.
\subsection{The setup}
Let $ \frac{1+\theta}{2} < \alpha < 1$ and $T\ge 2$. We let $2 \le X \le Y \le T^{O(1)}$ be parameters to be determined, and we set 
\[
	M_{X}(s) = \sum_{n_{j}\le X}\frac{\mu(n_{j})}{n_{j}^{s}}.
\]
Here $\mu$ is the M\"obius function associated with the system $(\MP, \MN)$: $\mu(n_{j}) = (-1)^{\nu}$ if $n_{j}$ is the product of $\nu$ distinct generalized primes, and $\mu(n_{j})=0$ otherwise, i.e.\ if $n_{j}$ is not squarefree\footnote{Strictly speaking $\mu$ should be defined as a function of $j$ instead of $n_{j}$, as unique factorization of Beurling integers need not hold. For example, we can have $\mu(p_{1}p_{3}) = 1$, $\mu(p_{2}^{2})=0$, even if $p_{1}p_{3}$ and $p_{2}^{2}$ happen to have the same numerical value. We will continue this slight abuse of notation, as sums over $\mu(n_{j})$ are always well-defined.}. Then
\[
	\zeta_{\MP}(s)M_{X}(s) = \sum_{j=1}^{\infty}\frac{a_{j}}{n_{j}^{s}},
\]
where $a_{1} = 1$, $a_{j} = 0$ for $1 < n_{j} \le X$, and $|a_{j}| \le d(n_{j})$ for $n_{j} > X$. Here $d(n_{j}) = \#\{(m, l): n_{m}n_{l} = n_{j}\}$ denotes the divisor function of the system $(\MP, \MN)$.

Using the inverse Mellin transform of the Gamma function
\[
	\e^{-1/x} = \frac{1}{2\pi\I}\int_{\kappa-\I\infty}^{\kappa+\I\infty}\Gamma(w)x^{w}\dif w, \quad \kappa > 0,
\]
we obtain (with $s=\sigma+\I t$)
\begin{equation}
\label{before contour switching}
	\e^{-1/Y} + \sum_{n_{j}>X}\frac{a_{j}\e^{-n_{j}/Y}}{n_{j}^{s}} = \frac{1}{2\pi\I}\int_{\kappa-\I\infty}^{\kappa+\I\infty}\zeta_{\MP}(s+w)M_{X}(s+w)\Gamma(w)Y^{w}\dif w, \quad \kappa > 1-\sigma.
\end{equation}
We assume that $\alpha \le \sigma < 1$, and shift the contour of integration to the line $\Re w = \frac{1+\theta}{2}-\sigma$. As $\zeta_{\MP}$ is polynomially bounded, this is allowed in view of the exponential decay in vertical strips of the Gamma function. We pick up two residues at $w = 1-s$ and $w = 0$ to get
\begin{multline*}
	\e^{-1/Y} + \sum_{n_{j}>X}\frac{a_{j}\e^{-n_{j}/Y}}{n_{j}^{s}} = AM_{X}(1)\Gamma(1-s)Y^{1-s} + \zeta_{\MP}(s)M_{X}(s)\\
	 + \frac{1}{2\pi}\int_{-\infty}^{\infty}\zeta_{\MP}\Bigl(\frac{1+\theta}{2}+\I(t+v)\Bigr)M_{X}\Bigl(\frac{1+\theta}{2}+\I(t+v)\Bigr)\Gamma\Bigl(\frac{1+\theta}{2}-\sigma+\I v\Bigr)Y^{\frac{1+\theta}{2}-\sigma+\I v}\dif v .
\end{multline*}

Let now $s = \rho = \beta + \I\gamma$ be a zero of $\zeta_{\MP}$ with $\beta \ge \alpha$ and $|\gamma| \le T$. Then 
\begin{multline}
	\e^{-1/Y} + \sum_{n_{j}>X}\frac{a_{j}\e^{-n_{j}/Y}}{n_{j}^{\rho}} = AM_{X}(1)\Gamma(1-\rho)Y^{1-\rho} \\
	+ \frac{1}{2\pi}\int_{-\infty}^{\infty}\zeta_{\MP}\Bigl(\frac{1+\theta}{2}+\I(\gamma+v)\Bigr)M_{X}\Bigl(\frac{1+\theta}{2}+\I(\gamma+v)\Bigr)\Gamma\Bigl(\frac{1+\theta}{2}-\beta+\I v\Bigr)Y^{\frac{1+\theta}{2}-\beta+\I v}\dif v.
	\label{6 terms}
\end{multline}
The first term on the right hand side of \eqref{6 terms} is in absolute value at most $1/6$, provided that $|\gamma| \ge C_{1}\log T$ for a suitable constant $C_{1}$. Also, setting $Z=C_{2}\log T$ for a sufficiently large constant $C_{2}$, we get
\[
	\abs{\frac{1}{2\pi}\int_{|v|\ge Z}\zeta_{\MP}\Bigl(\frac{1+\theta}{2}+\I(\gamma+v)\Bigr)M_{X}\Bigl(\frac{1+\theta}{2}+\I(\gamma+v)\Bigr)\Gamma\Bigl(\frac{1+\theta}{2}-\beta+\I v\Bigr)Y^{\frac{1+\theta}{2}-\beta+\I v}\dif v} \le \frac{1}{6}.
\] 
Both estimates are due to the exponential decay of $\Gamma$. Furthermore, if $Y$ is sufficiently large, we have
\[
	\abs[3]{\sum_{n_{j} > Y^{2}}\frac{a_{j}\e^{-n_{j}/Y}}{n_{j}^{\rho}}} \le \frac{1}{6} \quad \text{and} \quad \e^{1/Y} \ge \frac{5}{6}.
\]
We conclude that for a zero $\rho=\beta+\I\gamma$ with $\beta\ge\alpha$ and $C_{1}\log T \le |\gamma| \le T$ we either have
\begin{equation}
\label{type I}
	\abs[3]{\sum_{X<n_{j}\le Y^{2}}\frac{a_{j}\e^{-n_{j}/Y}}{n_{j}^{\rho}}} \ge \frac{1}{6},
\end{equation} 
or 
\begin{equation}
\label{type II}
	\abs{\int_{-Z}^{Z}\zeta_{\MP}\Bigl(\frac{1+\theta}{2}+\I(\gamma+v)\Bigr)M_{X}\Bigl(\frac{1+\theta}{2}+\I(\gamma+v)\Bigr)\Gamma\Bigl(\frac{1+\theta}{2}-\beta+\I v\Bigr)Y^{\frac{1+\theta}{2}-\beta+\I v}\dif v} \ge \frac{2\pi}{6}.
\end{equation}

From the set of all such $\zeta_{\MP}$-zeros, we select a subset $\mathcal{R}$ of $3Z$-well-spaced zeros (meaning that $|\gamma-\gamma'| \ge 3Z$ if $\beta+\I\gamma$ and $\beta'+\I\gamma'$ are distinct elements from $\mathcal{R}$) in such a way that 
\[
	N(\alpha, T) \ll (\log T)^{2} + |\mathcal{R}|(\log T)^{2}.
\]
This is possible in view of Lemma \ref{local number of zeros} and the trivial estimate $N(\frac{1+\theta}{2}, C_{1}\log T) \ll (\log T)^{2}$. Let $\mathcal{R}_{1}$ and $\mathcal{R}_{2}$ be the subsets of $\mathcal{R}$ for which \eqref{type I} and \eqref{type II} hold respectively, and denote their respective sizes by $R_{1}$ and $R_{2}$.

\subsection{The estimation of $R_{1}$}
To estimate $R_{1}$, we decompose each sum satisfying \eqref{type I} into $O(\log Y) = O(\log T)$ dyadic subsums. By the pigeonhole principle, there is a number $N$, $X\le N\le Y^{2}$, such that 
\[
	\abs[3]{\sum_{N<n_{j}\le 2N}\frac{a_{j}\e^{-n_{j}/Y}}{n_{j}^{\rho}}} \gg \frac{1}{\log T}
\]
holds for $\gg R_{1}(\log T)^{-1}$ zeros $\rho\in\mathcal{R}_{1}$. Hence we have
\begin{equation}
\label{R1 prelim}
	R_{1} \ll (\log T)^{3}\sum_{\rho\in\mathcal{R}_{1}}\abs[3]{\sum_{N<n_{j}\le 2N}\frac{a_{j}\e^{-n_{j}/Y}}{n_{j}^{\rho}}}^{2}.
\end{equation}
We now apply Theorem \ref{discrete MVT S} with $\eta = N^{\theta-1}$. In view of the $\theta$-well-behavedness \eqref{well-behaved}, $\chi(n_{j}, N^{\theta}) \ll N^{\theta}$ for $N < n_{j} \le 2N$. This yields
\begin{align*}
	R_{1}	&\ll (\log T)^{3}(T+N^{1-\theta})(\log T)\sum_{N<n_{j}\le 2N}N^{\theta}\frac{|a_{j}|^{2}\e^{-2n_{j}/Y}}{N^{2\alpha}} \\
			&\ll (TN^{1+\theta-2\alpha} + N^{2-2\alpha})(\log T)^{7}\e^{-2N/Y}.
\end{align*} 
In the last step we used the estimate $\sum_{n_{j}\le x}d(n_{j})^{2} \ll x(\log x)^{3}$, which follows from \cite[Proposition 4.4.1]{Knopfmacher} (in fact an asymptotic is shown there). As $N \ge X$, we conclude that 
\begin{equation}
\label{R1}
	R_{1} \ll (TX^{1+\theta-2\alpha} + Y^{2-2\alpha})(\log T)^{7}.
\end{equation}

\subsection{The estimation of $R_{2}$}
Let us now treat the zeros in $\mathcal{R}_{2}$. For each zero $\rho = \beta+\I\gamma$ satisfying \eqref{type II}, let $t_{\rho}$ be a real number with $|\gamma-t_{\rho}| \le Z$ for which $\abs[1]{\zeta_{\MP}\bigl(\frac{1+\theta}{2}+\I t_{\rho}\bigr)M_{X}\bigl(\frac{1+\theta}{2}+\I t_{\rho}\bigr)}$ is maximal on $[\gamma-Z, \gamma+Z]$. We now also assume that $\alpha \ge \frac{1+\theta}{2} + (\log T)^{-1}$, which we may, as our zero-density estimate follows from the trivial estimate $N(\alpha, T)\ll T\log T$ in the range $\alpha <  \frac{1+\theta}{2} + (\log T)^{-1}$. With this restriction on $\alpha$, 
\[
	\int_{-Z}^{Z}\abs[2]{\Gamma\Bigl(\frac{1+\theta}{2}-\beta+\I v\Bigr)}\dif v \ll \log T, \quad \text{whenever } \beta\ge \alpha.
\]
Hence, \eqref{type II} implies that for each $\rho\in \mathcal{R}_{2}$, 
\begin{equation}
\label{lower bound zetaM}
	\abs[2]{\zeta_{\MP}\Bigl(\frac{1+\theta}{2}+\I t_{\rho}\Bigr)M_{X}\Bigl(\frac{1+\theta}{2}+\I t_{\rho}\Bigr)} \gg Y^{\beta-\frac{1+\theta}{2}}(\log T)^{-1},
\end{equation}
from which we get that 
\begin{equation}
\label{R2 prelim}
	R_{2} \ll Y^{\frac{1+\theta}{2}-\alpha}(\log T)\sum_{\rho\in \mathcal{R}_{2}}\abs[2]{\zeta_{\MP}\Bigl(\frac{1+\theta}{2}+\I t_{\rho}\Bigr)M_{X}\Bigl(\frac{1+\theta}{2}+\I t_{\rho}\Bigr)}.
\end{equation}

From the $3Z$-well-spacedness of the numbers $\gamma$, it follows that the $t_{\rho}$ are $Z$-well-spaced. This allows us to estimate the sum of $\zeta_{\MP}$-values by an integral by an elementary lemma of Gallagher (see e.g.\ \cite[Lemma 1.2]{Montgomery}):
\begin{align}
	\sum_{\rho\in\mathcal{R}_{2}}\abs[2]{\zeta_{\MP}\Bigl(\frac{1+\theta}{2}+\I t_{\rho}\Bigr)}^{2} 
	&\ll \int_{-T-Z}^{T+Z}\biggl(\,\abs[2]{\zeta_{\MP}\Bigl(\frac{1+\theta}{2}+\I t\Bigr)}^{2} + \abs[2]{\zeta_{\MP}\Bigl(\frac{1+\theta}{2}+\I t\Bigr)\zeta'_{\MP}\Bigl(\frac{1+\theta}{2}+\I t\Bigr)}\biggr)\dif t \nonumber\\
	&\ll	T(\log T)^{2}, \label{sum zeta}
\end{align}
where in the last step we used Theorem \ref{MVT zeta} and Cauchy--Schwarz. 

For $M_{X}$, we perform a dyadic splitting and write $M_{X}(s) = 1 + \sum_{l=0}^{L}D_{l}(s)$, where $L = \lfloor\frac{\log X}{\log 2}\rfloor \ll \log T$ and
\[
	D_{l}(s) = \sum_{2^{l}<n_{j}\le 2^{l+1}}\frac{\mu(n_{j})}{n_{j}^{s}}, \quad l=0,1,\dotsc, L-1, \quad D_{L}(s) = \sum_{2^{L}< n_{j}\le X}\frac{\mu(n_{j})}{n_{j}^{s}}.
\]
Hence by Cauchy--Schwarz,
\begin{equation}
\label{dyadic subsums M}
	\sum_{\rho\in\mathcal{R}_{2}}\abs[2]{M_{X}\Bigl(\frac{1+\theta}{2}+\I t_{\rho}\Bigr)}^{2} 
	\ll R_{2} + \log T\sum_{l=0}^{L}\sum_{\rho\in\mathcal{R}_{2}}\abs[2]{D_{l}\Bigl(\frac{1+\theta}{2}+\I t_{\rho}\Bigr)}^{2}.
\end{equation}

We estimate the inner sum by Theorem \ref{discrete MVT T}, with $\eta = (2^{l+1})^{\theta-1}$, for which $\chi(n_{j}, \eta 2^{l+1}) \ll (2^{l})^{\theta}$ (when $2^{l} < n_{j}\le 2^{l+1})$. This yields
\begin{align*}
	\sum_{\rho\in\mathcal{R}_{2}}\abs[2]{D_{l}\Bigl(\frac{1+\theta}{2}+\I t_{\rho}\Bigr)}^{2} 
		&\ll \bigl(T+(2^{l})^{1-\theta}\bigr)(\log T)\sum_{2^{l}<n_{j}\le 2^{l+1}}(2^{l})^{\theta} \frac{1}{(2^{l})^{1+\theta}} \\
		&\ll \bigl(T+(2^{l})^{1-\theta}\bigr)(\log T).
\end{align*}
We conclude in view of the trivial estimate $R_{2}\le T$ that
\begin{equation}
\label{sum MX}
	\sum_{\rho\in\mathcal{R}_{2}}\abs[2]{M_{X}\Bigl(\frac{1+\theta}{2}+\I t_{\rho}\Bigr)}^{2} \ll T + (\log T)^{2}\sum_{l=0}^{L}\bigl(T+(2^{l})^{1-\theta}\bigr) \ll (T+ X^{1-\theta})(\log T)^{3}.
\end{equation}

Inserting the estimates \eqref{sum zeta} and \eqref{sum MX} in \eqref{R2 prelim} after another application of Cauchy--Schwarz yields
\begin{equation}
\label{R2}
	R_{2} \ll Y^{\frac{1+\theta}{2}-\alpha}\bigl(T+(TX^{1-\theta})^{1/2}\bigr)(\log T)^{7/2}.
\end{equation}

\subsection{Conclusion}
To conclude the proof of Theorem \ref{zero-density theorem}, we take $X = T^{\frac{1}{1-\theta}}$ and $Y = T^{\frac{2}{3-2\alpha-\theta}}$. Note that $X < Y$ and $\log Y \ll \log T$ uniformly in $\alpha$. Substituting these values in \eqref{R1} and \eqref{R2} yields
\[
	|\mathcal{R}| \le R_{1} + R_{2} \ll T^{\frac{4(1-\alpha)}{3-2\alpha-\theta}}(\log T)^{7},
\]
from which we finally obtain \eqref{main zero-density estimate}.
\begin{remark}
In \cite{BrouckeHilberdink}, also the (sharp) upper bounds $\frac{1}{T}\int_{0}^{T}\abs{\zeta_{\MP}(\sigma_{0}+\I t)}^{2}\dif t \ll T^{\frac{1+\theta-2\sigma_{0}}{1-\theta}}$, $\frac{1}{T}\int_{0}^{T}\abs{\zeta'_{\MP}(\sigma_{0}+\I t)}^{2}\dif t \ll T^{\frac{1+\theta-2\sigma_{0}}{1-\theta}}(\log T)^{2}$ for $\theta < \sigma_{0} < \frac{1+\theta}{2}$ were established. One can check that shifting to the line $\Re w = \sigma_{0}-\sigma$ instead of $\Re w = \frac{1+\theta}{2}-\sigma$ in \eqref{before contour switching} and using the above second moment estimates leads to inferior results.
\end{remark}

\section{Proof of Theorem \ref{general theorem}}
\label{sec: general}
The proof of the more general Theorem \ref{general theorem} follows without much effort along the same lines as in Section \ref{sec: proof}. We recall our setup: $(n_{j})_{j\ge1}$ is a sequence of Beurling integers satisfying $N(x+x^{\theta})-N(x-x^{\theta}) \ll_{\eps} x^{\theta+\eps}$, $F(s) = \sum_{j\ge1}a_{j}n_{j}^{-s}$ is a Dirichlet series with non-negative coefficients satisfying $a_{1}=1$, $a_{j}\ll_{\eps} n_{j}^{\eps}$, and $A(x) = xP(\log x) + O(x^{\theta}(\log x)^{q})$, and $G(s) = 1/F(s) = \sum_{j\ge1}b_{j}n_{j}^{-s}$ satisfies $\abs{b_{j}}\ll_{\eps}n_{j}^{\eps}$.
\begin{remark} 
The sequence of Beurling integers $(n_{j})_{j}$ is not necessarily strictly increasing and can have repeated values. In this situation, there is not a unique choice of coefficients $a_{j}$. For Theorem \ref{general theorem} it suffices that there is some choice of coefficients $a_{j}$ satisfying $a_{j}\ll_{\eps} n_{j}^{\eps}$. The same remark applies to $G$. Hence, if the integer $n_{j_{0}}$ has multiplicity $m_{j_{0}}$ (which can be up to $n_{j_{0}}^{\theta+\eps}$ on our assumptions), we can allow following bound for the ``total coefficient'': $\sum_{j: n_{j}=n_{j_{0}}}a_{j}\ll_{\eps} m_{j_{0}}n_{j_{0}}^{\eps}$, for every $\eps>0$.
\end{remark}

Note that $N(x+x^{\theta}) - N(x-x^{\theta}) \ll_{\eps} x^{\theta+\eps}$ implies that $N(x) \ll_{\eps} x^{1+\eps}$. In combination with the Ramanujan condition $\abs{b_{j}} \ll_{\eps} n_{j}^{\eps}$, this implies that $G(s)$ converges absolutely for $\Re s > 1$ and is analytic there. The absolute convergence of $F(s)$ for $\Re s > 1$ already followed from \eqref{A(x)}. From $F(s)G(s)=1$ it follows that $F$ has no zeros for $\Re s > 1$.

\subsection{Auxiliary lemmas}
We begin by establishing a bound for the local number of zeros of $F$ as in Lemma \ref{local number of zeros}. It results from combining Jensen's formula with a polynomial bound for $F(s)$.
\begin{lemma}
\label{polynomial bound}
Suppose $F(s) = \sum_{j}a_{j}n_{j}^{-s}$ has non-negative coefficients $a_{j}$ whose counting function $A(x)$ satisfies \eqref{A(x)}. Then for every $\delta>0$ we have uniformly for $\theta+\delta \le \sigma \le 1-\delta$, $\abs{t}\ge 2$:
\[
	F(\sigma+\I t) \ll_{\delta} \abs{t}^{\frac{1-\sigma}{1-\theta}}(\log \abs{t})^{\max\{d, q\}},
\]
and uniformly for $\sigma \ge \frac{1+\theta}{2}$, $\abs{t}\ge 2$:
\[
	F(\sigma+\I t) \ll \abs{t}^{\frac{1-\sigma}{1-\theta}}(\log \abs{t})^{\max\{d+1, q\}} + (\log \abs{t})^{d+1}.
\]
\end{lemma}
\begin{proof}
Writing $A(x) = xP(\log x) + E(x)$ with $\deg P = d$ and $E(x) \ll x^{\theta}(\log x)^{q}$, we have for $\sigma>1$ and fixed $x>1$,
\begin{align*}
	F(s) 	&= \sum_{n_{j}\le x}\frac{a_{j}}{n_{j}^{s}} + \int_{x}^{\infty}\frac{\dif\,\bigl(uP(\log u)\bigr)}{u^{s}} + \int_{x}^{\infty}\frac{\dif E(u)}{u^{s}}\\
		&= \sum_{n_{j}\le x}\frac{a_{j}}{n_{j}^{s}} + x^{1-s}\biggl(\frac{Q_{1}(\log x)}{1-s} + \dotsb + \frac{Q_{d+1}(\log x)}{(1-s)^{d+1}}\biggr) - \frac{E(x)}{x^{s}} + s\int_{x}^{\infty}\frac{E(u)}{u^{s+1}}\dif u.
\end{align*}
Here, the $Q_{j}$ are certain polynomials of degree at most $d+1-j$. For fixed $x$, the above expression has analytic continuation to $\Re s > \theta$ with the exception of $s=1$. The result follows from inserting the bounds
\[
	\abs[3]{\sum_{n_{j}\le x}\frac{a_{j}}{n_{j}^{s}}} \le \sum_{n_{j}\le x}\frac{a_{j}}{n_{j}^{\sigma}} \ll \frac{x^{1-\sigma}-1}{1-\sigma}(\log x)^{d}, \quad E(x) \ll x^{\theta}(\log x)^{q}	
\]
and by selecting $x = \abs{t}^{\frac{1}{1-\theta}}$. Note that we may assume $\abs{s} \ll \abs{t}$ as the result is trivial for $\sigma > 2$.
\end{proof}

\begin{lemma}
\label{local zeros F}
With the same assumptions as in the previous lemma, we have the local bound
\[
	N(F; \alpha, T+1) - N(F; \alpha, T) \ll_{\alpha} \log T.
\]
\end{lemma}
\begin{proof}
It suffices to show that for large $T$ and for every $\delta>0$ and $0 < r \le 2-\theta-\delta$, the number of zeros of $F$ in the disk $\abs{s-2-\I T}\le r$ is $O_{\delta}(\log T)$. Let $\rho_{1}, \dotsc, \rho_{K}$ be an enumeration of the zeros of $F$ in this disk, and set $r_{k} = \abs{\rho_{k}-2-\I T}$. Setting $R = r+\delta/2$, it follows from Jensen's formula that
\[
	\log\frac{R}{r_{1}} + \dotsb + \log\frac{R}{r_{K}} \le \frac{1}{2\pi}\int_{0}^{2\pi}\log\abs[1]{F(2+\I T+R\e^{\I\theta})}\dif\theta + \log\frac{1}{\abs{F(2+\I T)}}.
\]
The left hand side is $\gg_{\delta} K$, while the right hand side is $\le O_{\delta}(\log T)$, in view of Lemma \ref{polynomial bound} and 
\[
	\frac{1}{\abs{F(2+\I T)}} = \abs{G(2+\I T)} \le \sum_{j=1}^{\infty}\frac{\abs{b_{j}}}{n_{j}^{2}} < \infty.
\]
\end{proof}

The final preliminary result we require is the mean value estimate for $F$ and $F'$. Integrating by parts yields
\[
	\sum_{n_{j}\le x}a_{j}\log n_{j} = xP(\log x)\log x - \int_{1}^{x}P(\log u)\dif u + O\bigl(x^{\theta}(\log x)^{q+1}\bigr) = x\tilde{P}(\log x) + O\bigl(x^{\theta}(\log x)^{q+1}\bigr)
\]
for a certain polynomial $\tilde{P}$ of degree $d+1$. Therefore \cite[Theorem 1.1+]{BrouckeHilberdink} yields
\begin{equation}
\label{second moment F}
	\frac{1}{T}\int_{0}^{T}\abs[2]{F\Bigl(\frac{1+\theta}{2}+\I t\Bigr)}^{2}\dif t \ll (\log T)^{d+q+1}, 
	\quad \frac{1}{T}\int_{0}^{T}\abs[2]{F'\Bigl(\frac{1+\theta}{2}+\I t\Bigr)}^{2}\dif t \ll (\log T)^{d+q+3}.
\end{equation}

\subsection{Proof of Theorem \ref{general theorem}: required modifications}
Fix a small $\eps>0$. In the setup of the proof, we replace $M_{X}(s)$ by
\[
	G_{X}(s) = \sum_{n_{j}\le X}\frac{b_{j}}{n_{j}^{s}}
\]
and write $F(s)G_{X}(s) = \sum_{j\ge1}c_{j}n_{j}^{-s}$, where $c_{1}=1$, $c_{j}=0$ for $1<n_{j}\le X$, and $\abs{c_{j}} \ll_{\eps}n_{j}^{\eps}d(n_{j}) \ll_{\eps'}n_{j}^{\eps'}$. Here and in what follows, $\eps'$ denotes a (small) positive number, not necessarily the same at each occurrence, and which is at most a bounded multiple of $\eps$.

When switching contours in \eqref{before contour switching}, the residue at $w=1-s$ now not only involves the value of $G_{X}(s+w)\Gamma(w)Y^{w}$ at $w=1-s$, but also that of its derivatives up to the order $d$. Due to the exponential decay of $\Gamma$ and its derivatives, this residue, when $s= \rho=\beta+\I\gamma$ is a zero of $F$, is at most $1/6$ in absolute value if $\abs{\gamma} \ge C_{1}\log T$ for a suitable constant $C_{1}$. Due to Lemma \ref{local zeros F}, we may again select a suitable subset of a $3Z$-well-spaced zeros of $F$, and we similarly consider the subsets $\mathcal{R}_{1}$ and $\mathcal{R}_{2}$.

The estimation of $R_{1}$ can be done analogously to Section \ref{sec: proof}: starting from \eqref{R1 prelim} (with coefficients $c_{j}$ instead of $a_{j}$), the assumption $N(x+x^{\theta})-N(x-x^{\theta}) \ll_{\eps} x^{\theta+\eps}$ allows us to apply Theorem \ref{discrete MVT S} with $\eta = N^{\theta-1}$ and $\chi(n_{j}, N^{\theta}) \ll_{\eps} N^{\theta+\eps}$. We now simply bound $\sum_{N<n_{j}\le 2N}\abs{c_{j}}^{2}$ by $O_{\eps'}(N^{1+\eps'})$. Hence we obtain \eqref{R1}, with the modification of replacing the factor $(\log T)^{7}$ by $T^{\eps'}$.

The estimation of $R_{2}$ also goes through as before, using the second moment bounds \eqref{second moment F}. The dyadic subsums are now
\[
	D_{l}(s) = \sum_{2^{l}<n_{j}\le 2^{l+1}}\frac{b_{j}}{n_{j}^{s}},
\]
whose coefficients are $O_{\eps}(X^{\eps})$. Hence we arrive at \eqref{R2}, where $(\log T)^{7/2}$ is replaced by $T^{\eps'}$.

The selection of the parameters $X$ and $Y$ is unchanged, and combining the modified equations \eqref{R1} and \eqref{R2} we arrive at
\[
	N(F; \alpha, T) \ll_{\eps'} T^{\frac{4(1-\alpha)}{3-2\alpha-\theta}+\eps' }.
\]
\begin{remark}
\label{general coefficients}
The non-negativity of the coefficients $a_{j}$ and \eqref{A(x)} were used to obtain polynomial bounds for $F$ and the second moment estimates \eqref{second moment F}. These assumptions can be replaced by the following ones, without affecting the conclusion of Theorem \ref{general theorem}. Suppose $F(s) = 1 + \sum_{j=2}^{\infty}a_{j}n_{j}^{-s}$ is a general Dirichlet series with complex coefficients satisfying $\abs{a_{j}} \ll_{\eps} n_{j}^{\eps}$ for every $\eps>0$. Suppose $F(s)$ has analytic continuation, with the exception of a possible pole at $s=1$, to a half-plane $\Re s > \sigma_{f}$, $\sigma_{f}<\frac{1+\theta}{2}$, where it satisfies $F(\sigma+\I t) \ll_{\delta} \abs{t}^{O_{\delta}(1)}$ for $\sigma \ge \sigma_{f}+\delta$ and for every $\delta>0$, and the second moment estimates
\[
	\frac{1}{2T}\int_{-T}^{T}\biggl(\abs[2]{F\Bigl(\frac{1+\theta}{2}\Bigr)}^{2} + \abs[2]{F'\Bigl(\frac{1+\theta}{2}\Bigr)}^{2}\biggr)\dif t \ll_{\eps} T^{\eps}, \quad \text{for every } \eps > 0.
\]
\end{remark}
\section{Conditional improvements}
\label{sec: conditional improvements}

For the Riemann zeta function, we currently have a fourth moment estimate on the $1/2$-line, and this stronger piece of information yields a better zero-density estimate. In general, better growth bounds on the Beurling zeta function will lead to improvements in the zero-density estimates. In this section we investigate such conditional improvements. We focus on two kinds of growth bounds: \emph{higher order moment estimates} and \emph{subconvexity bounds}.
\begin{theorem}
\label{zero-density moments}
Let $(\MP, \MN)$ be a Beurling number system satisfying \eqref{well-behaved} and the moment estimates
\begin{align}
	\frac{1}{T}\int_{0}^{T}\abs[2]{\zeta_{\MP}\Bigl(\frac{1+\theta}{2}+\I t\Bigr)}^{2k}\dif t 	&\ll_{k} (\log T)^{O_{k}(1)}, \label{moment bound zeta}\\
	\frac{1}{T}\int_{0}^{T}\abs[2]{\zeta'_{\MP}\Bigl(\frac{1+\theta}{2}+\I t\Bigr)}^{2k}\dif t 	&\ll_{k} (\log T)^{O_{k}(1)}, \label{moment bound zeta'}
\end{align}
for some (real) $k\ge1$.
Then we have uniformly for $\alpha\ge\frac{1+\theta}{2}$,
\begin{equation}
\label{zero-density estimate moments}
	N(\alpha, T) \ll_{k} T^{\frac{2(k+1)(1-\alpha)}{2(1-\alpha)+k(1-\theta)}}(\log T)^{O_{k}(1)};
\end{equation}
and uniformly for $\alpha \ge \frac{3k-1+(k-1)\theta}{4k-2}$,
\begin{equation}
\label{HM moment}
	N(\alpha, T) \ll_{k} T^{\frac{4(3\alpha-2-\theta)(1-\alpha)}{k(2\alpha-1-\theta)(4\alpha-3-\theta)+4(3\alpha-2-\theta)(1-\alpha)}}(\log T)^{O_{k}(1)}.
\end{equation}
Furthermore, if $2k$ is an integer, the assumption \eqref{moment bound zeta'} may be omitted.
\end{theorem}
When $k=1$, \eqref{zero-density estimate moments} reduces to \eqref{main zero-density estimate}, while \eqref{HM moment} is stricty weaker.
A fourth moment estimate ($k=2$) would yield $N(\alpha,T) \ll T^{\frac{3(1-\alpha)}{2-\alpha-\theta}}(\log T)^{O(1)}$, which for $\theta=0$ corresponds to Ingham's zero-density estimate for the Riemann zeta function \cite{Ingham40}. When $\alpha \to 1$, the exponent of $T$ in the estimate \eqref{HM moment} is $\sim \frac{4(1-\alpha)}{k(1-\theta)}$.

 Any Beurling zeta function of $\theta$-well-behaved integers satisfies the ``convexity bound'' $\zeta_{\MP}(\sigma+\I t) \ll_{\sigma} \abs{t}^{\frac{1-\sigma}{1-\theta}}$ (see e.g. \cite[Lemma 2.6]{ReveszvonMangoldt} or Lemma \ref{polynomial bound}). On the $\frac{1+\theta}{2}$-line this gives a square root bound. Assuming a smaller pointwise bound, we obtain the following improvement. 
\begin{theorem}
\label{zero-density subconvexity}
Let $(\MP, \MN)$ be a Beurling number system satisfying \eqref{well-behaved} and the subconvexity bound
\begin{equation}
\label{subconvexity bound}
	\zeta_{\MP}\Bigl(\frac{1+\theta}{2}+\I t\Bigr) \ll_{B} \abs{t}^{B}
\end{equation}
 for some $0< B < 1/2$. Then uniformly for $\alpha \ge \frac{1+\theta}{2}$,
\begin{equation}
\label{Ingham subconvexity}
	N(\alpha, T) \ll_{B} T^{\frac{2(1+2B)(1-\alpha)}{1-\theta}}(\log T)^{9};
\end{equation}
and uniformly for $\alpha \ge \frac{3+\theta}{4}$,
\begin{equation}
\label{Montgomery-type}
	N(\alpha, T) \ll_{B} T^{\frac{8B(3\alpha-2-\theta)(1-\alpha)}{(2\alpha-1-\theta)(4\alpha-3-\theta)}}(\log T)^{O_{B}(1)}.
\end{equation}
\end{theorem}
For the Riemann zeta function, the conditional estimates \eqref{Ingham subconvexity} and \eqref{Montgomery-type} (with $\theta=0$ and up to some power of $\log T$) were first established by Ingham \cite{Ingham37} and Montgomery \cite{Montgomery69} respectively.
In view of $N(\alpha, T)\ll T\log T$, \eqref{Montgomery-type} is non-trivial only if $\alpha$ is some distance away from $\frac{3+\theta}{4}$. When $\alpha \to 1$, the exponent of $T$ in \eqref{Montgomery-type} is $\sim \frac{8B(1-\alpha)}{1-\theta}$. Note that for $\alpha$ close to $1$, \eqref{HM moment} and \eqref{Montgomery-type} are of similar quality when $B = \frac{1}{2k}$.
 
Assuming a form of the Lindel\"of Hypothesis\footnote{Note that $\forall \eps > 0: \zeta_{\MP}(\frac{1+\theta}{2}+\I t) \ll_{\eps} \abs{t}^{\eps}$ is equivalent with $\forall \eps > 0, k \ge 1: \frac{1}{T}\int_{0}^{T}\abs[1]{\zeta_{\MP}(\frac{1+\theta}{2}+\I t)}^{2k}\dif t \ll_{\eps, k} T^{\eps}$.} for the Beurling zeta function, we obtain:
\begin{corollary}
Assume that \eqref{subconvexity bound} holds with arbitrarily small $B$. Then we have for every $\eps>0$ and $\delta>0$, 
\[
	N(\alpha, T) \ll_{\eps} T^{\frac{2(1-\alpha)}{1-\theta}+\eps} \quad \text{and} \quad N\Bigl(\frac{3+\theta}{4}+\delta, T\Bigr) \ll_{\eps, \delta} T^{\eps}.
\]
\end{corollary}
In analogy with the Riemann zeta function, we refer to the statement ``$N(\zeta_{\MP}; \alpha, T) \ll_{\eps} T^{\frac{2(1-\alpha)}{1-\theta}+\eps}$ for every $\eps>0$'' as the \emph{density hypothesis} for the Beurling zeta function $\zeta_{\MP}(s)$.

\medskip

The proofs of \eqref{HM moment} and \eqref{Montgomery-type} require another technical tool, the Hal\'asz--Montgomery inequality, which we adapt to the Beurling setting. 

\subsection{The Hal\'asz--Montgomery inequality}
This inequality is an alternative for bounding discrete mean values of Dirichlet polynomials. It is based on a Hilbert space inequality, and allows one to exploit good bounds for (sums of) zeta sums $\sum_{n_{j}\le N}n_{j}^{-\I t}$ (or smoothed version thereof), if available. 
Higher order moment estimates and subconvexity bounds lead to such ``good bounds'' which are of sufficient quality to obtain an improvement of Theorem \ref{discrete MVT T} in certain ranges of the involved parameters.  

The basic inequality we shall use is \cite[Lemma 1.5]{Montgomery}, which states that
\begin{equation}
\label{Bombieri}
	\sum_{r=1}^{R}\abs{\langle \xi, \vphi_{r}\rangle}^{2} \le \norm{\xi}^{2}\max_{r}\sum_{s=1}^{R}\abs{\langle \vphi_{r}, \vphi_{s}\rangle},
\end{equation}
for vectors $\xi$, $\vphi_{r}$, $r=1, \dotsc, R$, in an inner product space.
We set $S(s) = \sum_{N<n_{j}\le 2N}a_{j}n_{j}^{-s}$, where $a_{j}$ are complex numbers and $n_{j}$ are Beurling integers of a number system satisfying \eqref{well-behaved}. Letting $\mathcal{T} = \{t_{r}: r=1, \dotsc, R\}\subseteq [0,T]$ be a set of $1$-well-spaced points, we apply \eqref{Bombieri} with 
\[
	\xi=\biggl(\frac{a_{j}\mathcal{I}_{N < n_{j}\le 2N}}{\sqrt{\e^{-n_{j}/(2N)}-\e^{-n_{j}/N}}}\biggr)_{j \ge 1}, \quad \vphi_{r} = \Bigl(\sqrt{\e^{-n_{j}/(2N)}-\e^{-n_{j}/N}}n_{j}^{-\I t_{r}}\Bigr)_{j \ge 1}.
\]
Here $\mathcal{I}_{N < n_{j}\le 2N} = 1$ if $N < n_{j} \le 2N$ and $0$ else.
Note that $\sqrt{\e^{-n_{j}/(2N)}-\e^{-n_{j}/N}} \asymp 1$ for $N< n_{j}\le 2N$; this smoothing factor is introduced to obtain faster convergence in the Mellin transform. The inequality \eqref{Bombieri} thus yields
\[
	\sum_{r=1}^{R}\abs{S(\I t_{r})}^{2} \ll \biggl(\sum_{N<n_{j}\le 2N}\abs{a_{j}}^{2}\biggr)\max_{r}\sum_{s=1}^{R}\abs{\sum_{j=1}^{\infty} \bigl(\e^{-n_{j}/(2N)}-\e^{-n_{j}/N}\bigr)n_{j}^{-\I(t_{r}-t_{s})}}.
\]
The inner-most sum can be written as
\[
	\frac{1}{2\pi\I}\int_{2-\I\infty}^{2+\I\infty}\zeta_{\MP}\bigl(w+\I(t_{r}-t_{s})\bigr)\bigl((2N)^{w}-N^{w}\bigr)\Gamma(w)\dif w.
\]
In this integral, we shift the contour to the line $\Re w = \frac{1+\theta}{2}$. We pick up a residue at $w=1-\I(t_{r}-t_{s})$, so that for each $r=1,\dotsc, R$,
\begin{multline}
\label{HM prelim}
	\sum_{s=1}^{R}\abs{\sum_{j=1}^{\infty} \bigl(\e^{-n_{j}/(2N)}-\e^{-n_{j}/N}\bigr)n_{j}^{-\I(t_{r}-t_{s})}} \\
		\ll N\sum_{s=1}^{R}\Gamma\bigl(1-\I(t_{r}-t_{s})\bigr) + N^{\frac{1+\theta}{2}}
		\sum_{s=1}^{R}\int_{-\infty}^{\infty}\abs{\zeta_{\MP}\Bigl(\frac{1+\theta}{2}+\I(v+t_{r}-t_{s})\Bigr)\Gamma\Bigl(\frac{1+\theta}{2}+\I v\Bigr)}\dif v .
\end{multline}

At this point we insert a pointwise or moment bound of $\zeta_{\MP}$. If we assume \eqref{subconvexity bound}, then it follows that
\begin{equation}
\label{MVT subconvex}
	\sum_{r=1}^{R}\abs{S(\I t_{r})}^{2} \ll_{B} \bigl(N + N^{\frac{1+\theta}{2}}T^{B}R\bigr)\sum_{N<n_{j}\le 2N}\abs{a_{j}}^{2},
\end{equation}
as the $1$-well-spacedness of $\mathcal{T}$ ensures that $\sum_{s=1}^{R}\Gamma\bigl(1-\I(t_{r}-t_{s})\bigr) \ll 1$.

\medskip
Suppose on the other hand that we have the moment bound \eqref{moment bound zeta} for some $k\ge 1/2$.
We truncate the integral in \eqref{HM prelim} to $\abs{\Im w} \le Z = C_{2}\log T$ for some sufficiently large constant $C_{2}$. 
Applying H\"older's inequality twice and using the 1-well-spacedness of $\mathcal{T}$ in \eqref{HM prelim}, we obtain
\begin{multline*}
	\sum_{s=1}^{R}\abs{\sum_{j=1}^{\infty} \bigl(\e^{-n_{j}/(2N)}-\e^{-n_{j}/N}\bigr)n_{j}^{-\I(t_{r}-t_{s})}}	
		\ll N + N^{\frac{1+\theta}{2}}\sum_{s=1}^{R}\int_{-Z}^{Z}\abs{\zeta_{\MP}\Bigl(\frac{1+\theta}{2}+\I(v+t_{r}-t_{s})\Bigr)}\dif v\\
		\ll N + N^{\frac{1+\theta}{2}}\biggl(\int_{-T-Z}^{T+Z}\abs{\zeta_{\MP}\Bigl(\frac{1+\theta}{2}+\I v\Bigr)}^{2k}\dif v\biggr)^{\frac{1}{2k}}R^{\frac{2k-1}{2k}}\log T.	
\end{multline*}
Hence we get
\begin{equation}
\label{MVT moment bound}
	\sum_{r=1}^{R}\abs{S(\I t_{r})}^{2} \ll_{k} \bigl(N + N^{\frac{1+\theta}{2}}T^{\frac{1}{2k}}R^{\frac{2k-1}{2k}}(\log T)^{O_{k}(1)}\bigr)\sum_{N<n_{j}\le 2N}\abs{a_{j}}^{2}.
\end{equation}
\begin{remark}
In any case, \eqref{subconvexity bound} always holds with $B=1/2$, while Theorem \ref{MVT zeta} yields \eqref{moment bound zeta} with $k=1$. Hence we obtain from \eqref{MVT moment bound}
\begin{equation}
\label{HM k=1}
	\sum_{r=1}^{R}\abs{S(\I t_{r})}^{2} \ll \bigl(N + N^{\frac{1+\theta}{2}}T^{1/2}R^{1/2}(\log T)^{O(1)}\bigr)\sum_{N<n_{j}\le 2N}\abs{a_{j}}^{2}.
\end{equation}
This estimate is superior to the estimate
\begin{equation}
\label{MVT compared to HM}
	\sum_{r=1}^{R}\abs{S(\I t_{r})}^{2} \ll(N+TN^{\theta})(\log N)\sum_{N<n_{j}\le 2N}\abs{a_{j}}^{2}	
\end{equation}
obtained from Theorem \ref{discrete MVT T} only if $R \ll TN^{\theta-1}$  and $N^{1-\theta} = o(T)$ (disregarding logarithmic factors). However, in this regime, both \eqref{HM k=1} and \eqref{MVT compared to HM} follow from the trivial estimate by Cauchy--Schwarz
\[
	\sum_{r=1}^{R}\abs{S(\I t_{r})}^{2} \ll RN\sum_{n_{j}\le N}\abs{a_{j}}^{2}.
\]
Thus, without posing additional assumptions on the Beurling zeta function besides well-behavedness of the integers, the Hal\'asz--Montgomery inequality does not seem to supply new information.
\end{remark}

We finally need the following modifications of \eqref{MVT subconvex} and \eqref{MVT moment bound}. Let $\{s_{r}: 1\le r \le R\}$ be a set of complex numbers satisfying $\Re s_{r} \ge \alpha$ and $\abs{\Im s_{r} - \Im s_{r'}} \ge 1$ whenever $r\neq r'$. Then \eqref{subconvexity bound} and \eqref{moment bound zeta} imply
\begin{equation}
\label{MVT S subconvex}
	\sum_{r=1}^{R}\abs{S(s_{r})}^{2} \ll_{B} \bigl(N + N^{\frac{1+\theta}{2}}T^{B}R\bigr)\sum_{N<n_{j}\le 2N}\frac{\abs{a_{j}}^{2}}{N^{2\alpha}}
\end{equation}
and 
\begin{equation}
\label{MVT S moment}
	\sum_{r=1}^{R}\abs{S(s_{r})}^{2} \ll_{k} \bigl(N + N^{\frac{1+\theta}{2}}T^{\frac{1}{2k}}R^{\frac{2k-1}{2k}}\bigr)(\log T)^{O_{k}(1)}\sum_{N<n_{j}\le 2N}\frac{\abs{a_{j}}^{2}}{N^{2\alpha}},
\end{equation}
respectively. This follows from \eqref{MVT subconvex} and \eqref{MVT moment bound} respectively, along the same lines as the proof of Theorem \ref{discrete MVT S}.

\subsection{Proofs}
We are now ready to prove Theorems \ref{zero-density moments} and \ref{zero-density subconvexity}.
\begin{proof}[Proof of Theorem \ref{zero-density moments}]
To prove \eqref{zero-density estimate moments}, we go straight to the estimation of $R_{2}$. Raising \eqref{lower bound zetaM} to the power $\frac{2k}{k+1}$, summing over $\mathcal{R}_{2}$, and applying H\"older's inequality gives
\begin{equation}
\label{R2 with higher moment}
	R_{2} \ll Y^{\frac{2k}{k+1}(\frac{1+\theta}{2}-\alpha)}\biggl(\sum_{\rho\in \mathcal{R}_{2}}\abs[2]{\zeta_{\MP}\Bigl(\frac{1+\theta}{2}+\I t_{\rho}\Bigr)}^{2k}\biggr)^{\frac{1}{k+1}}
	\biggl(\sum_{\rho\in \mathcal{R}_{2}}\abs[2]{M_{X}\Bigl(\frac{1+\theta}{2}+\I t_{\rho}\Bigr)}^{2}\biggr)^{\frac{k}{k+1}}(\log T)^{O(1)}.
\end{equation}
Applying Gallagher's lemma (\cite[Lemma 1.2]{Montgomery}), H\"older's inequality, and \eqref{moment bound zeta}-\eqref{moment bound zeta'}, we find 
\[
	\sum_{\rho\in \mathcal{R}_{2}}\abs[2]{\zeta_{\MP}\Bigl(\frac{1+\theta}{2}+\I t_{\rho}\Bigr)}^{2k} \ll_{k} T(\log T)^{O_{k}(1)}.
\]
 Inserting this estimate and \eqref{sum MX} into \eqref{R2 with higher moment} gives, upon selecting once again $X=T^{\frac{1}{1-\theta}}$, 
\begin{equation}
\label{R2 moments}
 	R_{2} \ll_{k} Y^{\frac{2k}{k+1}(\frac{1+\theta}{2}-\alpha)}T(\log T)^{O_{k}(1)}.
\end{equation}
Comparing with the estimate \eqref{R1} for $R_{1}$ now yields the optimal choice $Y = T^{\frac{k+1}{2(1-\alpha) + k(1-\theta)}}$, from which \eqref{zero-density estimate moments} follows.

To prove \eqref{HM moment}, we assume that $\alpha \ge \frac{3k-1+(k-1)\theta}{4k-2} (> \frac{3+\theta}{4})$. We apply \eqref{MVT S moment} to bound the right hand sides of \eqref{R1 prelim} and \eqref{dyadic subsums M}, and insert the obtained upper bound in the latter case into \eqref{R2 with higher moment}. After some calculations, again assuming $\log X \le \log Y \ll_{k} \log T$, we get
\begin{align}
	R_{1}	&\ll_{k} \bigl(Y^{2-2\alpha} + TX^{-k(4\alpha-3-\theta)}\bigr)(\log T)^{O_{k}(1)}, \label{R1'}\\
	R_{2}	&\ll_{k} \bigl(X^{\frac{k(1-\theta)}{k+1}}Y^{-\frac{k(2\alpha-1-\theta)}{k+1}}T^{\frac{1}{k+1}} + X^{\frac{k(1-\theta)}{3}}Y^{-\frac{2k(2\alpha-1-\theta)}{3}}T\bigr)(\log T)^{O_{k}(1)}, \label{R2'}
\end{align}
which can be compared with \eqref{R1} and \eqref{R2 moments} respectively. We now set
\[
	X= T^{\frac{2\alpha-1-\theta}{k(2\alpha-1-\theta)(4\alpha-3-\theta)+4(3\alpha-2-\theta)(1-\alpha)}} \quad \text{and} \quad
	Y= T^{\frac{2(3\alpha-2-\theta)}{k(2\alpha-1-\theta)(4\alpha-3-\theta)+4(3\alpha-2-\theta)(1-\alpha)}}.
\]
Note that $X < Y$ as $\alpha>\frac{3+\theta}{4}$, and $\log Y \ll_{k}\log T$. With these choices, both terms in the right hand side of \eqref{R1'} are equal, while the second term in the right hand side of \eqref{R2'} dominates the first (again since $\alpha>\frac{3+\theta}{4}$), and we obtain \eqref{HM moment}.

\medskip
When $2k$ is an integer, we may avoid Gallagher's lemma, which requires information on the derivative. Instead, we will bound the sum of well-spaced $\zeta_{\MP}^{2k}$-values by a lemma of Ivi\'c, \cite[Lemma 7.1]{Ivic}:
\begin{lemma}
Let $T$ be large. Uniformly for $\abs{t} \le T$, and with $Z = C_{2}\log T$, $C_{2}$ a sufficiently large constant (depending on $k$), we have
\[
	\abs[2]{\zeta_{\MP}\Bigl(\frac{1+\theta}{2}+\frac{1}{\log T}+\I t\Bigr)}^{2k} \ll_{k} 1 + \log T\int_{-Z/2}^{Z/2}\abs[2]{\zeta_{\MP}\Bigl(\frac{1+\theta}{2}+\I t+\I v\Bigr)}^{2k}\e^{-|v|}\dif v.
\]
\end{lemma} 
The lemma follows upon writing, with $s=\frac{1+\theta}{2}+\frac{1}{\log T} + \I t$, 
\[
	\frac{1}{2\pi\I}\int_{1-\I\infty}^{1+\I\infty}\zeta_{\MP}(s+w)^{2k}\Gamma(w)\dif w = \sum_{j=1}^{\infty}d_{2k}(n_{j})n_{j}^{-s}\e^{-n_{j}} \ll_{k} 1,
\]
 and shifting the contour to $\Re w = -\frac{1}{\log T}$. Here $d_{2k}$ is the $2k$-fold divisor function of the number system $(\MP, \MN)$. The switching of contours picks up two residues; for $w=0$ the residu is $\zeta_{\MP}(s)^{2k}$, and for $w=1-s$ the residu is $O_{k}(1)$, in view of the exponential decay of the gamma function and its derivatives. For the same reason, we can truncate the remaining integral to the interval $[-Z/2, Z/2]$, and the lemma then follows from inserting the estimate $\Gamma(u+\I v) \ll \e^{-|v|}\abs{u+\I v}^{-1}$, which holds uniformly for $v\in \R$ and $u$ in compacts.
 
 \medskip
 In the proof of the zero-density estimate, we now assume that $\alpha \ge \frac{1+\theta}{2}+\frac{2}{\log T}$ and shift the contour in \eqref{before contour switching} to $\Re w = \frac{1+\theta}{2}+\frac{1}{\log T} - \sigma$. We similarly obtain the estimate \eqref{R2 with higher moment}, but with $\frac{1+\theta}{2}$ replaced by $\frac{1+\theta}{2}+\frac{1}{\log T}$. Recalling that the $t_{\rho}$'s are $Z$-well-spaced, we may estimate the sum of the $\zeta_{\MP}^{2k}$-values by the $2k$-th moment of $\zeta_{\MP}$, in view of Ivi\'c' lemma, while we can keep the previously obtained bounds for the sum involving $M_{X}$ in both cases. As $Y \le T^{O_{k}(1)}$ we have $Y^{\frac{1}{\log T}} \ll_{k} 1$, so that we may complete the proof in the same way as before.
\end{proof}

\begin{proof}[Proof of Theorem \ref{zero-density subconvexity}]
To estimate $R_{2}$, we will use the subconvexity bound \eqref{subconvexity bound} instead of the second moment estimate to treat the contribution of $\zeta_{\MP}(s)$. For zeros $\rho = \beta+\I\gamma\in \mathcal{R}_{2}$, we let $t_{\rho}$ be a real number with $\abs{\gamma-t_{\rho}} \le Z$ for which $\abs[1]{M_{X}\bigl(\frac{1+\theta}{2}+\I t_{\rho}\bigr)}$ is maximal on $[\gamma-Z, \gamma+Z]$. We recall that $Z = C_{2}\log T$ for some large constant $C_{2}$. Employing the subconvexity bound \eqref{subconvexity bound} in \eqref{type II} yields
\begin{equation}
\label{MX subconvexity}
	\abs{M_{X}\Bigl(\frac{1+\theta}{2}+\I t_{\rho}\Bigr)} \gg_{B} Y^{\alpha-\frac{1+\theta}{2}}T^{-B}(\log T)^{-1}.
\end{equation}

To prove \eqref{Ingham subconvexity}, we combine \eqref{MX subconvexity} with \eqref{sum MX} to obtain 
\[
	R_{2} \ll_{B} Y^{1+\theta-2\alpha}T^{2B}(T+X^{1-\theta})(\log T)^{5}.
\]
Keeping the bound \eqref{R1} for $R_{1}$ and setting $X=T^{\frac{1}{1-\theta}}$ and $Y = T^{\frac{1+2B}{1-\theta}}$, we obtain \eqref{Ingham subconvexity}.

\medskip
To prove \eqref{Montgomery-type}, we may assume that $\alpha \ge \frac{3+\theta}{4} + \delta$ for some $\delta = \delta(B)>0$, because otherwise \eqref{Montgomery-type} is trivial. Applying \eqref{MVT S subconvex} to the right hand side of \eqref{R1 prelim}, we get that for some $N$ with $X\le N\le Y^{2}$,  
\[
	R_{1} \ll_{B} (\log T)^{3}\bigl(N + N^{\frac{1+\theta}{2}}T^{B}R_{1}\bigr)N^{1-2\alpha}(\log N)^{3} \ll_{B} \bigl(N^{2-2\alpha} + N^{\frac{3+\theta}{2}-2\alpha}T^{B}R_{1}\bigr)(\log T)^{6}.
\]
Here we used that $\log Y \ll_{B} \log T$, which will follow from our choice of $Y$ later on. If $X^{2\alpha-\frac{3+\theta}{2}} \ge T^{B}(\log T)^{7}$ say, then the second term involving $R_{1}$ is negligible, so that
\begin{equation}
\label{R1 subconvex}
	R_{1} \ll_{B} Y^{2-2\alpha}(\log T)^{6}.
\end{equation}

For $R_{2}$, we perform a dyadic subdivision of \eqref{MX subconvexity} to see that there is some $K \le X$ such that 
\[
	\abs{\sum_{K<n_{j}\le 2K}\frac{\mu(n_{j})}{n_{j}^{\frac{1+\theta}{2}+\I t_{\rho}}}} \gg_{B} Y^{\alpha-\frac{1+\theta}{2}}T^{-B}(\log T)^{-1}(\log X)^{-1}
\]
holds for $\gg R_{2}(\log X)^{-1}$ zeros. Applying \eqref{MVT S subconvex}  and $\log X \le \log Y \ll_{B}\log T$ yields
\begin{align*}
	R_{2}	&\ll_{B} \sum_{\rho\in\mathcal{R}_{2}}\abs{\sum_{K<n_{j}\le 2K}\frac{\mu(n_{j})}{n_{j}^{\frac{1+\theta}{2}+\I t_{\rho}}}}^{2} Y^{1+\theta-2\alpha}T^{2B}(\log T)^{5} \\
			&\ll \bigl(K + K^{\frac{1+\theta}{2}}T^{B}R_{2}\bigr)K^{1-(1+\theta)}Y^{1+\theta-2\alpha}T^{2B}(\log T)^{5} \\
			&\ll \bigl(X^{1-\theta}Y^{1+\theta-2\alpha}T^{2B} + X^{\frac{1-\theta}{2}}Y^{1+\theta-2\alpha}T^{3B}R_{2}\bigr)(\log T)^{5}.
\end{align*}
The second term involving $R_{2}$ is negligible if $Y^{2\alpha-1-\theta} \ge X^{\frac{1-\theta}{2}}T^{3B}(\log T)^{6}$ say, in which case we get
\begin{equation}
\label{R2 subconvex}
	R_{2} \ll_{B} X^{1-\theta}Y^{1+\theta-2\alpha}T^{2B}(\log T)^{5}.
\end{equation}

We now choose $X$ and $Y$ so that
\begin{align*}
	X^{2\alpha-\frac{3+\theta}{2}} &= T^{B}(\log T)^{7} \quad &&\text{and} \quad &Y^{2\alpha-1-\theta} &= X^{\frac{1-\theta}{2}}T^{3B}(\log T)^{6}, \quad \text{i.e.} \\
	X &= T^{\frac{2B}{4\alpha-3-\theta}}(\log T)^{O_{B}(1)} \quad &&\text{and}  \quad &Y &= T^{\frac{4B(3\alpha-2-\theta)}{(4\alpha-3-\theta)(2\alpha-1-\theta)}}(\log T)^{O_{B}(1)}.
\end{align*}
Note that $X < Y$ and $\log Y \ll_{B}\log T$, as we assumed that $\alpha > \frac{3+\theta}{4}+\delta(B)$. Inserting these values into \eqref{R1 subconvex} and \eqref{R2 subconvex} we obtain after some calculations that
\[
	R_{1} \ll _{B}T^{\frac{8B(3\alpha-2-\theta)(1-\alpha)}{(2\alpha-1-\theta)(4\alpha-3-\theta)}}(\log T)^{O_{B}(1)}, 
	\quad R_{2} \ll_{B} T^{\frac{4B(1-\alpha)}{4\alpha-3-\theta}}(\log T)^{O_{B}(1)}.
\]
As $\alpha > \frac{3+\theta}{4}$, the bound for $R_{1}$ is the dominant quantity, and we obtain \eqref{Montgomery-type}.
\end{proof}
\begin{remark}
It could happen that one has a subconvexity bound, not on the line $\sigma=\frac{1+\theta}{2}$, but for some $\sigma$ closer to $1$, maybe with $\sigma=\sigma(T)$ even depending on $T$. It should be possible to obtain zero-density estimates in terms of the quantity $\sup_{\,\abs{t}\le T, \sigma\ge\sigma_{0}}\abs{\zeta_{\MP}(\sigma+\I t)}$, by shifting the contour of integration in the proofs to the line $\sigma=\sigma_{0}$.
\end{remark}

\subsection{A conjecture in the style of Montgomery}
We shall later see in Subsection \ref{DMV-example} that both the moment estimate \eqref{moment bound zeta} with $k>1$ and the subconvexity bound \eqref{subconvexity bound} with $B<1/2$ may fail. Another path to the density hypothesis is by assuming a stronger mean value theorem for (general) Dirichlet polynomials. Inspired by Montgomery's conjecture \cite[Conjecture 9.2]{Montgomery} we may formulate:
\begin{conjecture}
\label{Montgomery conjecture}
Let $\MN = (n_{j})_{j}$ be a sequence of $\theta$-well-behaved Beurling integers, and let $a_{j}$ be complex numbers with $|a_{j}|\le1$. Then uniformly for $\nu\in [1,2]$ and for every $\eps>0$,
\[
	\int_{T_{0}}^{T_{0}+T}\abs[2]{\sum_{n_{j}\le N}a_{j}n_{j}^{-\I t}}^{2\nu}\dif t \ll_{\eps} (TN^{\nu\theta} + N^{\nu})N^{\nu+\eps}.
\] 
\end{conjecture}
The idea is that we may treat $(\sum_{n_{j}\le N}a_{j}n_{j}^{-\I t})^{\nu}$ as if it were a Dirichlet polynomial of length $N^{\nu}$, even when $\nu$ is not an integer.

From this one can again show that
\begin{equation}
\label{discrete Montgomery}
	\sum_{s\in\mathcal{S}}\abs[2]{\sum_{N<n_{j}\le2N}a_{j}n_{j}^{-s}}^{2\nu} \ll_{\eps} (TN^{\nu\theta} + N^{\nu})N^{\nu(1-2\alpha)+\eps},
\end{equation}
whenever $\mathcal{S}$ is a set of points $s=\sigma+\I t$, with $\sigma\ge\alpha$ and $t\in [-T,T]$ well-spaced, in the sense that for $s, s'\in \mathcal{S}$ with $s\neq s'$ we have $|t-t'|\ge1$.

This generalization of Montgomery's conjecture leads to the density hypothesis for Beurling zeta functions.
Indeed, following the proof of Section \ref{sec: proof}, we now fix some $\eps>0$ and set $X=T^{\eps}$. Then trivially $M_{X}(\frac{1+\theta}{2}+\I t_{\rho}) \ll T^{\eps}$. Inserting this bound in \eqref{type II} and letting $t_{\rho}$ now be such that $\abs[1]{\zeta_{\MP}\bigl(\frac{1+\theta}{2}+\I t_{\rho}\bigr)}$ is maximal, we may use the second moment estimate of $\zeta_{\MP}$ to see that
\[
	R_{2} \ll_{\eps} Y^{1+\theta-2\alpha}T^{1+\eps'}.
\]
We again use the convention that $\eps'$ denotes a (small) positive number, not necessarily the same at each occurrence, and which is at most a bounded multiple of the previously fixed $\eps$.

Again we may find an $N$ and $\gg R_{1}/\log T$ zeros $\rho$ for which 
\[
	\abs[3]{\sum_{N<n_{j}\le 2N}\frac{a_{j}\e^{-n_{j}/Y}}{n_{j}^{\rho}}} \gg \frac{1}{\log T}.
\]
We may in fact assume that $X\le N \le Y(\log Y)^{2}$ (and could also have done so in Section \ref{sec: proof}), as $\abs{\sum_{n_{j}> Y(\log Y)^{2}} a_{j}\e^{-n_{j}/Y}n_{j}^{-\rho}} \le 1/6$ if $Y$ is sufficiently large.
We first choose an integer $m$ so that $N^{m} \le Y(\log Y)^{2} < N^{m+1}$, and subsequently a number $\nu\in[1,2]$ so that $N^{m\nu} = Y(\log Y)^{2}$. Note that $m\ll_{\eps}1$. Applying then \eqref{discrete Montgomery} to the Dirichlet polynomials 
\[
	\biggl(\sum_{N<n_{j}\le 2N}\frac{a_{j}\e^{-n_{j}/Y}}{n_{j}^{\rho}}\biggr)^{m} = \sum_{N^{m}< n_{j}\le (2N)^{m}}\frac{b_{j}}{n_{j}^{\rho}},
\]
whose coefficients $b_{j}$ satisfy $\abs{b_{j}} \le d_{2m}(n_{j}) \ll_{\eps, \eps'} n_{j}^{\eps'}$, yields
\[
	R_{1} \ll_{\eps} T^{\eps'}(TY^{\theta} + Y)Y^{1-2\alpha+\eps'},
\]
from which
\[
	N(\alpha, T) \ll_{\eps} T^{\frac{2(1-\alpha)}{1-\theta}+\eps'},
\]
upon choosing $Y=T^{\frac{1}{1-\theta}}$.

\section{Limitations for zero-density estimates}
\subsection{Sharpness of the range of validity $\alpha\ge\frac{1+\theta}{2}$}
First we indicate why it is not possible to obtain a zero-density estimate beyond the line $\sigma=\frac{1+\theta}{2}$. In the case $\theta=0$, this is evident from the Riemann zeta function $\zeta$, for which it is known that $N(\zeta;1/2, T) \sim \frac{T}{2\pi}\log\frac{T}{2\pi}$.

In the case $\theta>0$, we consider a suitable rescaling of $\zeta(s)$. Set
\begin{equation}
\label{zeta_theta}
	\zeta_{\theta}(s) \coloneqq \zeta\Bigl(\frac{s-\theta}{1-\theta}\Bigr) = \sum_{n=1}^{\infty}n^{\frac{\theta}{1-\theta}}n^{-\frac{s}{1-\theta}}
	= \exp \biggl(\sum_{p}\sum_{m=1}^{\infty}\frac{p^{\frac{m\theta}{1-\theta}}}{m}p^{-\frac{ms}{1-\theta}}\biggr), \quad \Re s > 1.
\end{equation}
We want to see this as the Beurling zeta function of a number system; such a system would have as its generalized integers the numbers $n^{\frac{1}{1-\theta}}$, $n\in\N$, with ``multiplicity'' $n^{\frac{\theta}{1-\theta}}$. Unfortunately, there is not an actual sequence of primes $\MP$ which yields such integers: for starters the ``multiplicities'' $n^{\frac{1}{1-\theta}}$ are in general not integral.
However, the definition of Beurling number systems may be generalized as follows: a Beurling number system \emph{in the extended sense} is a pair of non-decreasing right-continuous unbounded functions $(\Pi(x), N(x))$, satisfying $\Pi(1)=0$, $N(1)=1$, and linked via the identity
\[
	\zeta(s) \coloneqq \int_{1^{-}}^{\infty}x^{-s}\dif N(x) = \exp\int_{1}^{\infty}x^{-s}\dif \Pi(x)
\]
whenever the integrals converge. Equivalently, 
\[
	\dif N(x) = \exp^{\ast}(\dif \Pi(x)) = \delta_{1}(x) + \sum_{m=1}^{\infty}\frac{1}{m!}(\dif\Pi(x))^{\ast m},
\]
where $\delta_{y}(x)$ denotes the Dirac delta measure concentrated at $y$, and $(\dif F)^{\ast m}$ denotes the $m$-th convolution power with respect to the multiplicative convolution of measures $\ast$, which is defined as 
\[
	\int_{E}\dif F \ast \dif G = \iint_{uv\in E}\dif F(u)\dif G(v).
\]
This convolution naturally generalizes Dirichlet convolution.
We refer to \cite[Chapters 2--3]{DiamondZhangbook} for more background on this notion.

In our case, we set
\[
	\dif \Pi_{\theta}(x) = \sum_{p}\sum_{m=1}^{\infty}\frac{p^{\frac{m\theta}{1-\theta}}}{m}\delta_{p^{\frac{m}{1-\theta}}}(x), 
	\quad \dif N_{\theta}(x) = \sum_{n=1}^{\infty}n^{\frac{\theta}{1-\theta}}\delta_{n^{\frac{1}{1-\theta}}}(x).
\]
This system has associated zeta function $\zeta_{\theta}(s)$, given by \eqref{zeta_theta}. Note that 
\begin{equation}
\label{density N_theta}
	N_{\theta}(x) = \sum_{n^{\frac{1}{1-\theta}}\le x}n^{\frac{\theta}{1-\theta}} = (1-\theta)x + O(x^{\theta}),
\end{equation}
and 
\[
	N\Bigl(\zeta_{\theta}; \frac{1+\theta}{2}, T\Bigr) \sim \frac{(1-\theta)T}{2\pi}\log\frac{(1-\theta)T}{2\pi},
\]
so that for this system in the extended sense, no zero density estimate beyond the ``critical line'' $\sigma = \frac{1+\theta}{2}$ is possible.

\medskip

We will now construct a sequence of Beurling primes $\MP$ for which $\zeta_{\MP}(s)$ has the same zeros as $\zeta_{\theta}(s)$, and for which the integers are also $\theta$-well-behaved. For this however, we will assume that $\theta\ge1/2$.

First, we define a ``prime-counting'' function $\pi_{\theta}(x)$ for which $\Pi_{\theta}(x) = \sum_{m=1}^{\infty}\frac{\pi_{\theta}(x^{1/m})}{m}$. By M\"obius inversion we obtain, with $\mu$ now denoting the classical M\"obius function,
\[
	\dif\pi_{\theta}(x) = \sum_{k=1}^{\infty}\frac{\mu(k)}{k}\dif\Pi_{\theta}(x^{1/k}) = \sum_{p}\sum_{m=1}^{\infty}a_{p,m}\delta_{p^{\frac{m}{1-\theta}}}(x),
\]
where
\[
	a_{p,m} \coloneqq \frac{1}{m}\sum_{k\mid m}\mu(k)\bigl(p^{\frac{m\theta}{1-\theta}}\bigr)^{1/k}.
\]
\begin{lemma}
Let $m\in\N$ and $u>1$. Then
\[
	\sum_{k\mid m}\mu(k)u^{1/k} > 0.
\]
\end{lemma}
\begin{proof}
The lemma is trivial in case that $m=1$, so we assume $m\ge2$. Set $f_{0}(u) = \sum_{k\mid m}\mu(k)u^{1/k}$. As $f_{0}(1) = \sum_{k\mid m}\mu(m)=0$, it suffices to show that
$f_{0}'(u) > 0$ for $u>1$. Write 
\[
	f_{0}'(u) = \sum_{k\mid m}\frac{\mu(k)}{k}u^{1/k-1} \eqqcolon \frac{f_{1}(u)}{u}.
\]
Now we indeed have $f_{1}(u) > 0$ for $u\ge 1$: 
\[
	f_{1}(1) = \sum_{k\mid m}\frac{\mu(k)}{k} = \prod_{p\mid m}\Bigl(1-\frac{1}{p}\Bigr) > 0,
\]
and 
\[
	f_{1}'(u) = \sum_{k\mid m}\frac{\mu(k)}{k^{2}}u^{1/k-1} \ge 1 - \sum_{k=2}^{\infty}\frac{1}{k^{2}} > 0.
\]

\end{proof}

By the lemma, the coefficients $a_{p,m}$ are positive and $\dif \pi_{\theta}(x)$ is a positive measure. We now construct our sequence of generalized primes $\MP$ as the sequence of numbers $p^{\frac{m}{1-\theta}}$ with multiplicity $\lfloor a_{p,m}\rfloor$. Hence
\[
	\dif\pi_{\MP}(x) = \sum_{p}\sum_{m=1}^{\infty}\lfloor a_{p,m}\rfloor \delta_{p^{\frac{m}{1-\theta}}}(x).
\]
Note that $\pi_{R}(x) \coloneqq \pi_{\theta}(x) - \pi_{\MP}(x) \ge 0$. Setting $M =M_{x} = \lfloor (1-\theta)\frac{\log x}{\log 2}\rfloor$, we have the upper bound
\[
	\pi_{R}(x) \le \sum_{p^{\frac{m}{1-\theta}}\le x}1 = \pi(x^{1-\theta}) + \pi(x^{\frac{1-\theta}{2}}) + \dotsb + \pi(x^{\frac{1-\theta}{M}}) \le \Li(x^{1-\theta}) + O\Bigl(\frac{x^{1-\theta}}{(\log x)^{2}}\Bigr),
\]
by a weak form of the prime number theorem with remainder. Consequently, we also have
\begin{equation}
\label{bound Pi_R}
	\Pi_{R}(x) \coloneqq \sum_{m=1}^{\infty}\frac{\pi_{R}(x^{1/m})}{m} \le  \Li(x^{1-\theta}) + O\Bigl(\frac{x^{1-\theta}}{(\log x)^{2}}\Bigr).
\end{equation}

We will now consider $\pi_{R}(x)$ as the ``prime-counting function'' of a system in the extended sense. The system $(\MP, \MN)$ is then obtained by removing the ``primes'' of the system given by $\pi_{R}(x)$ from the ``primes'' of the system given by $\pi_{\theta}$. Consequently, we may calculate the integer-counting function $N_{\MP}(x)$ of the system $(\MP, \MN)$ as a convolution of $\dif N_{\theta}$ with the ``M\"obius function'' of $\pi_{R}$. 

Let $N_{R}(x)$ be the ``integer-counting'' function associated with $\Pi_{R}(x)$: $\dif N_{R}(x) = \exp^{\ast}(\dif\Pi_{R}(x))$, or equivalently, $N_{R}(x)$ is the unique function  with $N_{R}(1)= 1$ and for which
\[
	\zeta_{R}(s) \coloneqq \int_{1^{-}}^{\infty}x^{-s}\dif N_{R}(x) = \exp\int_{1}^{\infty}x^{-s}\dif\Pi_{R}(x).
\] 
Using a criterium for so-called ``O-density'' of integers by Diamond and Zhang, we show:
\begin{lemma}
\label{bound N_R}
We have the upper bound
\[
	N_{R}(x) \ll x^{1-\theta}.
\]
\end{lemma}
\begin{proof}
In order to apply the criterium of Diamond and Zhang, we first perform a rescaling and set
\[	
	\tilde{\Pi}(x) \coloneqq \Pi_{R}(x^{\frac{1}{1-\theta}}), \quad \dif\tilde{N}(x) \coloneqq \exp^{\ast}(\dif\tilde{\Pi}(x)).
\]
Then $\tilde{N}(x) = N_{R}(x^{\frac{1}{1-\theta}})$. Setting $\tilde{\psi}(x) \coloneqq \int_{1}^{x}\log u \dif\tilde{\Pi}(u)$, the bound \eqref{bound Pi_R} yields
\[
	\tilde{\psi}(x) \le x + O\Bigl(\frac{x}{\log x}\Bigr). 
\]
By \cite[Prop.\ 6.11]{DiamondZhangbook}, it follows that $\tilde{N}(x) \ll x$.
\end{proof}

Next, we set $\dif M_{R}(x) \coloneqq (\dif N_{R}(x))^{\ast-1}$, the convolution inverse of $\dif N_{R}(x)$. The function $M_{R}(x)$ is the natural analogue of the summatory function of the M\"obius function and we have 
\[
	\int_{1^{-}}^{\infty}x^{-s}\dif M_{R}(x) = \frac{1}{\zeta_{R}(s)}.
\]
As $\abs{\dif M_{R}(u)} = \abs{\exp^{\ast}(-\dif\Pi_{R}(u))} \le \exp^{\ast}(\dif\Pi_{R}(u)) = \dif N_{R}(u)$ and $N_{R}(x)\ll x^{1-\theta}$, the above integral converges absolutely for $\Re s > 1-\theta$.
The counting measure of the integers of our system $(\MP, \MN)$ is now given as the convolution $\dif N_{\MP}(x) = \dif N_{\theta}(x) \ast \dif M_{R}(x)$. On the zeta-side this equation translates to $\zeta_{\MP}(s) = \zeta_{\theta}(s)/\zeta_{R}(s)$. We now estimate this convolution:
\begin{align*}
	N_{\MP}(x) 	&= \int_{1^{-}}^{x}\dif N_{\theta}\ast\dif M_{R} = \int_{1^{-}}^{x}N_{\theta}(x/u)\dif M_{R}(u)\\
				&=\int_{1^{-}}^{x}\Bigl((1-\theta)x/u+O\bigr((x/u)^{\theta}\bigr)\Bigr)\dif M_{R}(u)\\
				&= \frac{(1-\theta)x}{\zeta_{R}(1)} + O\biggl(x\int_{x}^{\infty}\frac{\dif N_{R}(u)}{u} + x^{\theta}\int_{1^{-}}^{x}u^{-\theta}\dif N_{R}(u)\biggr),
\end{align*}
where in the second line we used \eqref{density N_theta}, and in the third line we used $\abs{\dif M_{R}(u)} \le \dif N_{R}(u)$. Integrating by parts, using Lemma \ref{bound N_R}, and recalling that we assumed $\theta \ge 1/2$, we see that the above error term is 
\[
	\ll x^{1-\theta} + x^{\theta}\int_{1}^{x}u^{-2\theta}\dif u,
\]
which is $\ll x^{\theta}$ if $\theta>1/2$, and $\ll x^{1/2}\log x$ if $\theta=1/2$. Hence we obtain
\[
	N_{\MP}(x) = \frac{1-\theta}{\zeta_{R}(1)}x + O\bigl(x^{\theta}(1+ \mathcal{I}_{\theta=1/2}\log x)\bigr),
\]
where $\mathcal{I}_{\theta=1/2}=1$ if $\theta=1/2$ and $0$ else.

Finally, 
\[
	\zeta_{\MP}(s) = \frac{\zeta_{\theta}(s)}{\zeta_{R}(s)} = \zeta_{\theta}(s)\exp\biggl(-\int_{1}^{\infty}x^{-s}\dif\Pi_{R}(x)\biggr).
\]
From \eqref{bound Pi_R} we see that $\int_{1}^{\infty}x^{-s}\dif\Pi_{R}(x)$ is analytic in $\Re s > 1-\theta$, so that in particular, 
\[
	N\Bigl(\zeta_{\MP}; \frac{1+\theta}{2}, T\Bigr) = N\Bigl(\zeta_{\theta}; \frac{1+\theta}{2}, T\Bigr) \sim \frac{(1-\theta)T}{2\pi}\log\frac{(1-\theta)T}{2\pi}.
\]
We have thus proved:
\begin{theorem}
Let $\theta \in [1/2,1)$. Then there exists a Beurling number system $(\MP, \MN)$ for which 
\[
	N_{\MP}(x) = Ax + O\bigl(x^{\theta}(1+\mathcal{I}_{\theta=1/2}\log x)\bigr)
\]
for some $A>0$, and whose zeta function $\zeta_{\MP}(s)$ satisfies
\[
	N\Bigl(\zeta_{\MP}; \frac{1+\theta}{2}, T\Bigr) \sim \frac{(1-\theta)T}{2\pi}\log\frac{(1-\theta)T}{2\pi}.
\]
\end{theorem}
\subsection{The extremal Diamond--Montgomery--Vorhauer-example}
\label{DMV-example}
Consider a density estimate of the form
\[
	N(\zeta_{\MP};\alpha,T) \ll_{\eps, \MP} T^{\frac{c_{1}(\alpha)(1-\alpha)}{1-\theta}+\eps}, \quad \frac{1+\theta}{2} \le \alpha \le 1,
\]
for every $\eps>0$ and every number system $(\MP, \MN)$ with $\theta$-well-behaved integers. It is an interesting question what the optimal value for $c_{1}(\alpha)$ is. Theorem \ref{zero-density theorem} shows that $c_{1}(\alpha) = \frac{4(1-\theta)}{3-2\alpha-\theta}$, which increases from $2$ to $4$ as $\alpha$ increases from $\frac{1+\theta}{2}$ to $1$, is admissible.

Regarding lower bounds, Diamond, Montgomery, and Vorhauer  \cite{DiamondMontgomeryVorhauer} showed that $c_{1}(\alpha) \ge 1$, at least when $\theta>1/2$. In their paper, they constructed for every $\theta \in (1/2, 1)$ a Beurling number system $(\MP_{\mathrm{DMV}}, \MN_{\mathrm{DMV}})$ with $N_{\mathrm{DMV}}(x) = Ax + O(x^{\theta})$ for some $A>0$ and for which $\zetaDMV(s)$ has infinitely many zeros on the contour $\sigma = 1 - \frac{a}{\log(|t|+2)}$, for some constant $a>0$. They showed that, as a result, no better remainder than the de la Vall\'ee-Poussin remainder $O\bigl(x\exp(-2\sqrt{a\log x})\bigr)$ can be obtained in the prime number theorem \eqref{PNT}. The authors also compute that for fixed $\alpha>\theta$, 
\[
	N(\zetaDMV; \alpha, T) = \Omega(T^{2\lambda(1-\alpha)}).
\]
Here $\lambda < \frac{1}{2(1-\theta)}$ is a parameter occurring in the definition of the number system which can be chosen arbitrarily close to $\frac{1}{2(1-\theta)}$. (In \cite{DiamondMontgomeryVorhauer}, this parameter is called $\alpha$.) Taking $\lambda = \frac{1}{2(1-\theta)}$ yields
\[
	N(\zetaDMV; \alpha, T) = \Omega\bigl(T^{\frac{1-\alpha}{1-\theta}}\bigr),
\]
at the cost of only having the slightly weaker $N_{\mathrm{DMV}}(x) = Ax + O_{\eps}(x^{\theta+\eps})$, for every $\eps>0$. (Probably the error term $O\bigl(x^{\theta}\exp(c(\log x)^{2/3})\bigr)$ with some $c>0$ is possible, see \cite[Chapter 17]{DiamondZhangbook} for the case $\theta=1/2$.) It would be interesting to construct examples of Beurling zeta functions, or even general Dirichlet series as in the statement of Theorem \ref{general theorem}, which exhibit $c_{1}(\alpha) > 1$ or perhaps even $c_{1}(\alpha) \ge 2$. 

One possible limiting factor to the amount of zeros of $\zetaDMV$ is the fact that it is an approximation of a ``continuous system''; more precisely, $\zetaDMV(s) = \zeta_{\mathrm{c}}(s)\e^{Z(s)}$, where $\zeta_{\mathrm{c}}(s)$ is the Mellin--Stieltjes transform of an \emph{absolutely continuous} measure $\dif N_{\mathrm{c}}$ (and $Z(s)$ satisfies $Z(\sigma+ \I t) \ll_{\sigma} \sqrt{\log(\,\abs{t}+2)}$ for $\sigma>1/2$). One might expect that Mellin--Stieltjes transforms of the form $\int_{1}^{\infty}x^{-s}\dif f(x) = \int_{1}^{\infty}x^{-s}f'(x)\dif x$ have in general fewer zeros than (discrete) Dirichlet series, for instance because of the fact that the mean values of $\int_{N}^{2N}x^{-s}f'(x)\dif x$ are better behaved (assuming $f$ is sufficiently ``nice'') than those of $\sum_{N< n_{j}\le 2N}a_{j}n_{j}^{-s}$.

\medskip
The zeta function $\zetaDMV(s)$ (with $\lambda=\frac{1}{2(1-\theta)}$) exhibits some other interesting extremal properties, which can be shown via routine calculations which we omit. For fixed $\sigma \in (\theta, 1)$, it satisfies
\[
	\zetaDMV(\sigma+\I t) =\Omega_{\eps}\bigl(\,\abs{t}^{\frac{1-\sigma}{1-\theta}-\eps}\bigr), \quad \text{ for every } \eps>0.
\]
This shows that the classical convexity bound $\zeta_{\MP}(\sigma+\I t) \ll_{\sigma} \abs{t}^{\frac{1-\sigma}{1-\theta}}$, $\theta<\sigma<1$ is essentially sharp for Beurling zeta functions $\zeta_{\MP}(s)$, and that the subconvexity bound \eqref{subconvexity bound} with some $B < 1/2$ may fail to hold.
 
Furthermore, letting $k > 1/2$ and $\theta < \sigma < 1 - \frac{1-\theta}{2k} = \frac{2k-1+\theta}{2k}$, one may verify that 
\[
	\frac{1}{T}\int_{0}^{T}\abs{\zetaDMV(\sigma+\I t)}^{2k}\dif t = \Omega_{\eps}\bigl(T^{\frac{2k-1+\theta-2k\sigma}{1-\theta}-\eps}\bigr), \quad \text{ for every } \eps>0.
\]
Setting $k=1$ gives that the condition $\sigma\ge\frac{1+\theta}{2}$ for having at most logarithmic growth of the second moment $\frac{1}{T}\int_{0}^{T}\abs{\zeta_{\MP}(\sigma+\I t)}^{2}\dif t$ in the class of $\theta$-well-behaved Beurling zeta functions $\zeta_{\MP}(s)$, $\theta>1/2$, is sharp. (In \cite{BrouckeHilberdink}, the range $\sigma\ge\frac{1+\theta}{2}$ was already shown to be sharp in a broader class of general Dirichlet series and for any $0\le \theta < 1$.) The Diamond--Montgomery--Vorhauer example also shows that higher order moment estimates \eqref{moment bound zeta} with $k>1$ are not available in general. 

Even if one is able to show a logarithmic bound for a higher order moment for all Beurling zeta functions on a line to the right of $\sigma = \frac{1+\theta}{2}$, this example shows that this line should be so close to $1$ that it most likely will not lead to an improvement in the zero-density estimate, at least not via the method of Section \ref{sec: proof}. More precisely, say that one is able to prove that for some $\sigma_{1} \in (\theta,1)$ and for some $k>1$, every Beurling zeta function $\zeta_{\MP}$ with $\theta$-well-behaved integers satisfies
\[
	\frac{1}{T}\int_{0}^{T}\abs[1]{\zeta_{\MP}\bigl(\sigma_{1}+\I t\bigr)}^{2k}\dif t \ll (\log T)^{O(1)}.
\]
The zeta function $\zetaDMV$ of Diamond, Montgomery, and Vorhauer then implies that $\sigma_{1} \ge 1-\frac{1-\theta}{2k}$. If one would like to exploit this $2k$-th moment, one would shift the contour in \eqref{before contour switching} to the line $\Re w = \sigma_{1} -\sigma$. Estimating $R_{2}$ similarly as in the proof of Theorem \ref{zero-density moments} would now only give
\[
	R_{2} \ll Y^{\frac{2k}{k+1}(\sigma_{1}-\alpha)}T^{\frac{1}{k+1}}\bigl(T + X^{2-2\sigma_{1}}\bigr)^{\frac{k}{k+1}}(\log T)^{O(1)},
\]
which is non-trivial only if $\alpha > \sigma_{1}$.
Choosing any $X \in [T^{\frac{1}{1-\theta}}, T^{\frac{1}{2-2\sigma_{1}}}]$ and comparing with the estimate \eqref{R1} of $R_{1}$ leads, in view of $\sigma_{1} \ge 1-\frac{1-\theta}{2k}$, to an estimate which is at best
\[
	N(\zeta_{\MP}; \alpha, T) \ll T^{\frac{2(k+1)(1-\alpha)}{3-2\alpha-\theta}}(\log T)^{O(1)},
\]
and hence inferior to the one provided by Theorem \ref{zero-density theorem}.

\subsection*{Funding} This work was supported by the Research Foundation -- Flanders [grant number 12ZZH23N].

\subsection*{Acknowledgements} The author would like to thank Gregory Debruyne for his helpful comments.

\end{document}